\documentclass[12pt]{article}
\usepackage[margin=0.8in]{geometry}
\usepackage{algcompatible}
\usepackage{caption}
\providecommand{\keywords}[1]{\small \quad \quad \textbf{Keywords:} #1}
\usepackage{titlesec}
\usepackage{hyperref}
\hypersetup{colorlinks=true, linkcolor=blue, filecolor=magenta, urlcolor=cyan, citecolor=blue}
\usepackage[affil-it]{authblk}
\usepackage{graphicx}
\usepackage{multirow}
\usepackage{amsmath,amssymb,amsfonts}
\usepackage{amsthm}
\usepackage{mathrsfs}
\usepackage[title]{appendix}
\usepackage{xcolor}
\usepackage{textcomp}
\usepackage{manyfoot}
\usepackage{booktabs}
\usepackage{longtable}
\usepackage{nicematrix}
\usepackage[numbers,sort&compress]{natbib}
\usepackage[linesnumbered,ruled,vlined]{algorithm2e}

\SetCommentSty{AlgoCommentStyle}
\usepackage{algorithmicx}
\usepackage{algpseudocode}
\usepackage{listings}
\usepackage{xspace}
\usepackage{dsfont}
\usepackage{cleveref}
\usepackage{amsmath}
\usepackage{tikz}
\tikzstyle{vertex}=[draw,circle,minimum size=18pt,inner sep=0pt]
\usepackage[font=small, labelfont=bf, margin=1cm]{caption}
\usepackage{subcaption}
\usepackage{tabularx}
\usepackage{url}
\usepackage{makecell}
\usepackage{orcidlink}

\newcommand{\CCP}{{CCP}\xspace}
\newcommand{\CCPs}{{CCPs}\xspace}
\newcommand{\MILP}{{MILP}\xspace}
\newcommand{\MILPs}{{MILPs}\xspace}
\newcommand{\LP}{{LP}\xspace}

\newcommand{\T}{\top}
\newcommand{\X}{\mathcal{X}}
\newcommand{\R}{\mathds{R}}

\newcommand{\BnC}{B\&C\xspace}

\newcommand{\stt}{\,:\,}

\DeclareMathOperator{\proj}{Proj}

\newcommand{\solver}[1]{\textsc{#1}\xspace}
\newcommand{\scipversion}{9.1.0}
\newcommand{\scip}{\solver{SCIP}}
\newcommand{\scipv}{\solver{SCIP}~\scipversion\xspace}

\newcommand{\soplexversion}{7.1.0}
\newcommand{\soplexv}{\solver{SoPlex}~\soplexversion\xspace}
\newcommand{\CB}{\mathcal{B}}
\newcommand{\CN}{\mathcal{N}}
\newcommand{\CC}{\mathcal{C}}
\newcommand{\FC}{\mathcal{C}}
\newcommand{\M}{\mathcal{M}}
\newcommand{\CL}{\mathcal{L}}
\newcommand{\CR}{\mathcal{R}}
\newcommand{\CO}{\mathcal{O}}
\newcommand{\I}{\mathcal{I}}
\newcommand{\J}{\mathcal{J}}
\newcommand{\F}{\mathcal{F}}
\newcommand{\MINLP}{{MINLP}\xspace}

\newcommand{\A}{\mathcal{A}}
\newcommand{\CS}{\mathcal{S}}
\newcommand{\ie}{i.e.,\ }

\newcommand{\rev}[1]{{\color{black}{#1}}}

\newcommand{\nbsc}[1]{\mbox{#1}\xspace}
\newcommand{\lsp}{\nbsc{CCLS}}
\newcommand{\rpp}{\nbsc{CCRP}}
\newcommand{\mpp}{\nbsc{CCMPP}}

\newcommand{\tblRF}{\texttt{B\&C+MIX+DI}\xspace}
\newcommand{\tblSRF}{\texttt{B\&C+MIX+sDI}\xspace}

\newcommand{\tblFBS}{\texttt{DB}\xspace}
\newcommand{\tblPBS}{\texttt{DB+OPF}\xspace}
\newcommand{\tblENPF}{\texttt{DB+EOPF}\xspace}

\newcommand{\tblRPP}{\texttt{CCRP}\xspace}

\newcommand{\tblMPP}{\texttt{CCMPP}\xspace}
\newcommand{\tblLSP}{\texttt{CCLS}\xspace}

\newcommand{\tblS}{\texttt{S}\xspace}
\newcommand{\tblT}{\texttt{T}\xspace}
\newcommand{\tblN}{\texttt{N}\xspace}
\newcommand{\tblDR}{\texttt{\%DP}\xspace}
\newcommand{\tblTDR}{\texttt{\texttt{\%NDI}}\xspace}
\newcommand{\tblcutoff}{\texttt{PN}\xspace}
\newcommand{\tbldomred}{\texttt{F}\xspace}
\newcommand{\tblTd}{\texttt{PT}\xspace}

\newcommand{\DP}{\texttt{\#DP}\xspace}
\newcommand{\MDP}{\texttt{\#MDP}\xspace}
\newcommand{\NUM}{\texttt{\#}}
\newcommand{\NDI}{\texttt{\#NDI}\xspace}
\newcommand{\PROB}{\texttt{Probs}}
\newcommand{\tblMIX}{\texttt{B\&C+MIX}\xspace}
\newcommand{\DeltaS}{\texttt{$\Delta$S}\xspace}
\newcommand{\RT}{\texttt{RT}\xspace}
\newcommand{\RN}{\texttt{RN}\xspace}

\theoremstyle{plain}
\newtheorem{Theorem}{Theorem}[section]
\newtheorem{Corollary}[Theorem]{Corollary}
\newtheorem{Lemma}[Theorem]{Lemma}
\newtheorem{Remark}[Theorem]{Remark}
\newtheorem{Example}[Theorem]{Example}

\newtheorem{Proposition}[Theorem]{Proposition}
\crefname{Theorem}{Theorem}{Theorems}
\crefname{Example}{Example}{Examples}
\crefname{Observation}{Observation}{Observations}
\crefname{Remark}{Remark}{Remarks}
\crefname{Proposition}{Proposition}{Propositions}
\crefname{Lemma}{Lemma}{Lemmas}
\crefname{Corollary}{Corollary}{Corollaries}
\crefname{subsection}{Section}{Sections}
\crefname{algorithm}{Algorithm}{Algorithms}
\crefname{figure}{Figure}{Figures}
\crefname{table}{Table}{Tables}
\crefname{section}{Section}{Sections}

\title{Exploiting Overlap Information in Chance-constrained Program with Random Right-hand Side}

\author[a]{Wei Lv\orcidlink{0009-0009-6861-9532}}
\author[b]{Wei-Kun Chen\orcidlink{0000-0003-4147-1346}}
\author[c,d]{Yu-Hong Dai\orcidlink{0000-0002-6932-9512}}
\author[a]{Xiao-Jiao Tong}

\affil[a]{\small School of Mathematics and Computational Science, Xiangtan University, Xiangtan 411105, China\\\textit{lvwei@xtu.edu.cn; dysftxj@hnfnu.edu.cn}}
\affil[b]{\small School of Mathematics and Statistics, Beijing Institute of Technology, Beijing 100081, China\\\textit{chenweikun@bit.edu.cn}}
\affil[c]{\small State Key Laboratory of Mathematical Sciences, Academy of Mathematics and Systems Science, Chinese Academy of Sciences, Beijing 100190, China}
\affil[d]{\small School of Mathematical Sciences, University of Chinese Academy of Sciences, Beijing 100049, China\\\textit{dyh@lsec.cc.ac.cn}}
\date{\small \today}

\begin{document}
	
\maketitle

\begin{abstract}
We consider the chance-constrained program (\CCP) with random right-hand side under a finite discrete distribution.
It is known that the standard mixed integer linear programming (\MILP) reformulation of the \CCP is generally difficult to solve by general-purpose solvers as the branch-and-cut search trees are enormously large, partly due to the weak linear programming relaxation.
In this paper, we identify another reason for this phenomenon: the intersection of the feasible regions of the subproblems in the search tree could be nonempty, leading to a wasteful duplication of effort in exploring the uninteresting overlap in the search tree.
To address the newly identified challenge and enhance the capability of the \MILP-based approach in solving \CCPs, 
we first show that the overlap in the search tree can be completely removed by a family of valid nonlinear if-then constraints, and then propose two practical approaches to tackle the highly nonlinear if-then constraints.
In particular, we use the concept of dominance relations between different scenarios of the random variables, and propose a novel branching, called dominance-based branching, which is able to create a valid partition of the problem with a much smaller overlap than the classic variable branching. 
Moreover, we develop overlap-oriented node pruning and variable fixing techniques, applied at each node of the search tree, to remove more overlaps in the search tree.
Computational results demonstrate the effectiveness of the proposed dominance-based branching \rev{with the} overlap-oriented node pruning and variable fixing techniques in reducing the search tree size and improving the overall solution efficiency.
	\vspace{8pt} \\
\end{abstract}

\section{Introduction}
\label{sect:introduction}

Consider the chance-constrained program (\CCP) with random right-hand side:
\begin{align}
	\label{ccps}
	\tag{CCP}
	\min \left \{ c^\top x \, : \,  \mathbb{P} \{Tx \geq \xi\} \geq 1-\epsilon,~ x \in \X \right \},
\end{align}
where $\X \subseteq \R^d$ is a polyhedron,
$T$ is an $m \times d$ matrix, $c$ is a $d$-dimensional cost vector, 
$\xi$ is an $m$-dimensional \emph{random} vector, 
and $\epsilon\in(0,1)$ is a confidence parameter chosen by the decision maker, typically near zero.
\eqref{ccps} is a powerful paradigm to model risk-averse decision-making problems, and has been seen in or served as a building block in many industrial applications in areas such as 
appointment scheduling \cite{Gurvich2010},
energy \cite{Wang2012,Qiu2015,Dey2023},
finance \cite{Dentcheva2004}, 
healthcare \cite{Beraldi2004},
facility location \cite{Beraldi2002a,Saxena2010,Chen2024},
supply chain logistic \cite{Lejeune2007,Murr2000,Beraldi2002}, 
and telecommunication \cite{Dentcheva2000}.
We refer to the surveys of  \citet{Prekopa2003} and \citet{Kucukyavuz2022} and the references therein for more applications of \eqref{ccps}.

\eqref{ccps} was introduced by  \citet{Charnes1963}, \citet{Miller1965}, \citet{Prekopa1970}, and \citet{Prekopa1973}, and has been extensively investigated in the literature \cite{Prekopa2003,Kucukyavuz2022}.  
For \eqref{ccps} with discrete random variables, 
\citet{Sen1992} derived a relaxation problem using disjunctive programming techniques.
\citet{Dentcheva2000} used the so-called $(1-\epsilon)$-efficient points \cite{Prekopa1990} to develop various reformulations for \eqref{ccps} with discrete random variables, and derived lower and upper bounds for the optimal value of the problem. 
Using a partial enumeration of the $(1-\epsilon)$-efficient points, \citet{Beraldi2002} proposed a specialized branch-and-bound algorithm for \eqref{ccps} that is based on the relaxation in \cite{Dentcheva2000} and guaranteed to find an optimal solution.
\citet{Lejeune2012a} presented a pattern-based solution method for obtaining a global solution of \eqref{ccps}, which requires the enumeration of the so-called   $(1-\epsilon)$-sufficient and -insufficient points \cite{Lejeune2012b} and solving an equivalent mixed integer linear programming (\MILP) reformulation.
\citet{Cheon2006} considered the case with a finite discrete distribution of the random vector $\xi$, that is, $\xi$ takes values $\xi^1, \xi^2,\ldots, \xi^n $ with $\mathbb{P}\{\xi = \xi^i\}= p_i \geq 0$ for $i\in [n]:=\{1,2, \ldots, n\}$ and $\sum_{i=1}^n p_i = 1$.
\citet{Cheon2006} provided a specialized branch-and-cut (\BnC), called branch-reduce-cut, algorithm that finds a global solution of \eqref{ccps} by successively partitioning the nonconvex feasible region and using bounding techniques.

\rev{In this paper, we are primarily interested in the \MILP reformulation-based approaches \cite{Ruszczynski2002,Luedtke2010a,Vielma2012,Kucukyavuz2012,Abdi2016,Zhao2017,Ahmed2017,Klnc-Karzan2022,Jiang2025}, which allow for the use of sophisticated \MILP solvers such as Gurobi and CPLEX to find an optimal solution for \eqref{ccps}.}
These approaches also rely on the assumption that $\xi$ has a finite discrete distribution $\xi^1,\xi^2, \ldots, \xi^{n}$, and require to solve an \MILP reformulation of \eqref{ccps}, detailed as follows.
Without loss of generality, we assume $\xi^i \geq 0$ for all $i \in [n]$ (by applying the transformation in \cite{Luedtke2010a,Kucukyavuz2022} if needed).
By introducing for each $i \in [n]$, a binary variable $z_i$, where $z_i = 0$ guarantees $v=Tx \geq \xi^i$,
\eqref{ccps} can then be reformulated as the following \MILP formulation \cite{Ruszczynski2002}:
\begin{equation}\label{milp}
	\min\left\{ c^\top x \, : \, \eqref{milp-consx}-\eqref{milp-consbound} \right\},\tag{MILP}
\end{equation}
where
\begin{align}
	\quad &Tx = v,~x \in \X, \label{milp-consx}\\
	\quad &v \geq \xi^{i}(1-z_i), ~\forall~i\in [n],  \label{milp-consvz}\\
	\quad &\sum_{i = 1}^n p_i z_i \leq \epsilon, \label{milp-consknap}\\
	\quad &v \in \R_+^m, ~z \in \{0,1\}^n. \label{milp-consbound}
\end{align}
Constraints \eqref{milp-consvz} are referred to as big-$M$ constraints where the big-$M$ coefficients are given by $\xi^i_k$, $k \in [m]$.
It is worthy noting that the assumption on the finite discrete distribution of $\xi$ is not restrictive, as for arbitrary distributions of $\xi$, we can use sample average approximation or importance sampling technique to obtain an approximation problem that satisfies this assumption; see \cite{Ahmed2008,Luedtke2008,Nemirovski2006a,Barrera2016}. 

It is well-known that due to the presence of the big-$M$ constraints in \eqref{milp-consvz}, the linear programming (\LP) relaxation of formulation \eqref{milp} is generally very weak, making it difficult to use state-of-the-art \MILP solvers to tackle formulation \eqref{milp} directly.
To bypass this difficulty, various approaches have been developed in the literature \cite{Luedtke2010a,Vielma2012,Kucukyavuz2012,Abdi2016,Zhao2017,Ahmed2017,Klnc-Karzan2022,Jiang2025}.
Specifically, 
using the lower bounds of variables $v$ (also known as quantile information \cite{Dentcheva2001,Lejeune2008}), 
\citet{Luedtke2010a} developed a strengthened version of  formulation \eqref{milp} that provides a tighter \LP relaxation and has less constraints. 
By investigating the single mixing set with a knapsack constraint, an important substructure of \eqref{milp} defined by constraints \eqref{milp-consvz} for a fixed $k \in [m]$ and \eqref{milp-consknap}, 
\citet{Luedtke2010a} employed the mixing inequalities \cite{Atamturk2000a,Gunluk2001b} to strengthen the \LP relaxation of formulation \eqref{milp}. 
\citet{Kucukyavuz2012}, \citet{Abdi2016}, and \citet{Zhao2017} further investigated the polyhedral structure of the single mixing set and proposed new families of inequalities to strengthen the \LP relaxation of formulation \eqref{milp}.
\citet{Kucukyavuz2012}, \citet{Zhao2017}, and \citet{Klnc-Karzan2022} investigated the joint mixing set with a knapsack constraint, that considers all constraints in \eqref{milp-consvz} and \eqref{milp-consknap}, and developed various aggregated mixing inequalities.
Computational evidences in \cite{Luedtke2010a,Kucukyavuz2012,Abdi2016,Zhao2017,solverscip2021} have demonstrated the effectiveness of the mentioned inequalities in enhancing the capability of employing \MILP solvers in solving formulation \eqref{milp}.
\citet{Luedtke2010a}, \citet{Vielma2012}, \citet{Kucukyavuz2012}, and \citet{Ahmed2017} developed various extended formulations that provide tighter \LP relaxations than the direct \LP relaxation of formulation \eqref{milp}.
\rev{Recently, \citet{Jiang2025} proposed a variable fixing approach that can reduce the problem size and tighten the \LP relaxation of formulation \eqref{milp}. 
Applying this technique requires computing an upper bound for the optimal value of formulation \eqref{milp} and solving single-scenario problems of the form $\min\{c^\top x \,:\, Tx = v,~x \in \X,~v \geq \xi^{i},~v \in \mathbb{R}_+^m\}$, $i \in [n]$.}

In another line of research, \citet{Ruszczynski2002} used the concept of dominance relations between scenarios and developed a class of valid inequalities, called dominance inequalities.
\rev{In \cite{Cheon2006}, the authors stated}
that the incorporation of dominance inequalities can \rev{reduce the solution time} 
of employing \MILP solvers in solving formulation \eqref{milp} (see \cite{Beraldi2010,Song2013,Song2014,Henrion2022} for similar discussions in other contexts){. However,} it was unknown until now how these inequalities improve the performance of \MILP solvers\rev{; in particular, whether they can improve the \LP relaxation of formulation \eqref{milp} is not addressed in the literature.}
As a byproduct of analysis, this paper closes this research gap by showing that adding the dominance inequalities into formulation \eqref{milp} cannot improve the \LP relaxation but can improve the performance of \MILP solvers by saving the computational efforts spent in exploring the uninteresting part of the search tree; see \cref{subsect:existing work}.
	
The goal of this paper is to explore new integer programming techniques to further improve the computational performance of the \MILP-based approach to solving \eqref{ccps}.
\subsection{Contributions and outline}
\label{subsect:contributions and outline}

Unlike existing approaches \cite{Luedtke2010a,Vielma2012,Kucukyavuz2012,Abdi2016,Zhao2017,Ahmed2017,Klnc-Karzan2022,Jiang2025} that mainly focus on improving the \LP relaxation of formulation \eqref{milp}, we go in a different direction by identifying another drawback of solving formulation \eqref{milp} using \BnC solvers; 
that is, the intersection of the feasible regions of the subproblems (after removing the common fixed variables) in the \BnC search tree could be nonempty,  leading to a wasteful duplication of effort in exploring the uninteresting overlap in the search tree. 
To overcome the drawback and enhance the capability of the \MILP-based approach in solving \eqref{ccps}, 
we first show that the overlap can be completely removed by adding a family of valid nonlinear if-then constraints into formulation \eqref{milp}, and then propose two practically tractable approaches to tackle the highly nonlinear if-then constraints.
More specifically, 
\begin{itemize}
	\item [$\bullet$] By using the concept of dominance relations between different scenarios of the random variables \cite{Ruszczynski2002}, we propose a novel branching, called dominance-based branching, which is able to create a valid partition of the current problem with a much smaller overlap than the classic variable branching. 
	We show that applying the proposed dominance-based branching to formulation \eqref{milp} is equivalent to applying the classic variable branching to formulation \eqref{milp} with the dominance inequalities in \cite{Ruszczynski2002} in terms of fixing the same variables and sharing the same \LP relaxation bound at two nodes with the same branching variables.
	We also employ a preprocessing technique to derive more dominance relations between the scenarios of the random variables, thereby leading to a much more effective dominance-based branching.
	\item [$\bullet$] By considering the joint mixing set with a knapsack constraint and the newly proposed if-then constraints, we develop node pruning and variable fixing techniques, applied at each node of the search tree, to remove more overlaps in the search tree. 
	Although implementing the proposed overlap-oriented  node pruning and variable fixing techniques is proved to be strongly NP-hard, we are able to develop an approximation algorithm that is competitive with the exact algorithm in terms of reducing the tree size while still enjoying an efficient polynomial-time worst-case complexity. 
\end{itemize}

We embed the proposed dominance-based branching \rev{with the} overlap-oriented node pruning and variable fixing techniques into the state-of-the-art open source \MILP solver \scip, and apply the resultant approach to solve the chance-constrained versions of the resource planning  (\rpp) problem \cite{Luedtke2014,Gurvich2010}, multiperiod power planning (\mpp) problem \cite{Dey2023}, and lot-sizing (\lsp) problem \cite{Beraldi2002}.
{Extensive computational results show that the two proposed approaches can significantly reduce the \BnC search tree size and substantially enhance the capability of \scip in solving \CCPs.}

This paper is organized as follows.
\cref{subsec:notations} presents the notations and assumptions used in this paper.
\cref{sect:weakness} shows that applying the classic \BnC algorithm to formulation \eqref{milp} leads to overlaps in the search tree, and presents a family of valid nonlinear if-then constraints to remove the overlaps. 
\cref{sect:precedence branch} develops the dominance-based branching and analyzes its relation to the direct use of the dominance inequalities in \cite{Ruszczynski2002}.  
\cref{sect:consistency branch} derives overlap-oriented node pruning and variable fixing techniques and provides the complexity and algorithmic design for the implementation.  
\cref{sect:computational results} reports the computational results. 
Finally, \cref{sect:Conclusion and future work} draws the conclusion.

\subsection{Notations and assumptions}
\label{subsec:notations}

For a nonnegative integer $n$, let $[n]=\{1, 2,\ldots, n\}$ where $[n] = \varnothing$ if $n =0$.
Let $\boldsymbol{0}$ denote the all zeros vector with an appropriate dimension.
For two vectors $v^1, v^2$ of dimension $m$, $v^1 \geq v^2$ denotes that $v^1_k \geq v^2_k$ holds for all $k \in [m]$; and $v^1 \ngeq v^2$ denotes that $v^1_k < v^2_k$ holds for at least one $k \in [m]$.
For vectors $\xi^1, \xi^2, \ldots, \xi^n$ and a subset $\CS \subseteq [n]$, we denote $\xi^{\CS}=\max_{i\in \CS}\{ \xi^i \}$, where the max is taken component-wise and $\xi^{\CS} = \boldsymbol{0}$ if $\CS=\varnothing$. 
We follow \cite{Margot2002,Ostrowski2011,Pfetsch2019} to characterize the current node of the search tree by $(\CB_0, \CB_1)$ or $(\CN_0, \CN_1)$, where $\CB_0$ and $\CB_1$ are the index sets of variables $z$  that have been branched on $0$ and $1$, respectively; and $\CN_0$ and $\CN_1$ are the index sets of variables $z$ that have been fixed to $0$ and $1$, respectively, by variable branching or other methods like reduced cost fixing or the coming overlap-oriented variable fixing.

For simplicity, we assume that variables $x$ in \eqref{ccps} are all continuous variables. 
However, our proposed approach can be applied to the general case, in which some of the variables $x$ may require to be integers, as long as the branching is allowed to be performed on variables $z_i$, $i \in [n]$, at a node.

\section{Overlaps arising in solving formulation \eqref{milp}}
\label{sect:weakness}
In this section, we first illustrate the weakness of applying the \BnC algorithm with the classic variable branching to solve  formulation \eqref{milp}; that is, the intersection of the feasible regions of the subproblems (after removing the common fixed variables) in the \BnC search tree could be nonempty,
 leading to a wasteful duplication of effort in exploring the uninteresting overlap in the search tree.
Then we strengthen formulation \eqref{milp} by presenting a family of valid nonlinear \emph{if-then} constraints to remove the overlaps during the \BnC process.
\subsection{Overlaps in the search tree} 
\label{subsect:overlaps}

We first apply the classic variable branching to formulation \eqref{milp}.
Without loss of generality, we 
suppose that variable $z_j$, $j \in [n]$, is branched on at the root node, and the two branches are
\begin{align}
	O^{\text{L}}=\min\left\{c^\T x \stt \eqref{milp-consx}-\eqref{milp-consbound},~z_j=0 \right\}, \label{trad_leftbranch}\\
	O^{\text{R}}=\min\left\{c^\T x \stt \eqref{milp-consx}-\eqref{milp-consbound},~z_j=1 \right\}.\label{trad_rightbranch}
\end{align}
By removing the common fixed variable $z_j$ from problems \eqref{trad_leftbranch} and \eqref{trad_rightbranch}, we obtain
\begin{align}
	\label{T-leftbranch}
	& O^{\text{L}}=\min\left\{ c^\T x \stt (x,v,z) \in \CC,~v \geq \xi^{j},
	~\sum_{i \in [n]\backslash\{j\}} p_i z_i \leq \epsilon \right\},\\
	\label{T-rightbranch}
	& O^{\text{R}} = \min\left\{c^\T x \stt (x,v,z) \in \CC, ~v \geq \boldsymbol{0},
	~p_j+\sum_{i \in [n]\backslash\{j\}} p_i z_i \leq \epsilon \right\},
\end{align}
where 
\begin{equation*}
	\CC = \left\{(x,v,z) \in \X\times\R_{+}^{m}\times\{0,1\}^{n-1} \stt Tx = v, ~v \geq \xi^{i}(1-z_i), 
	~\forall~i \in [n]\backslash\{j\} \right\}
\end{equation*}
represents the feasible set of formulation \eqref{milp} with constraints $v \geq {\xi^j}(1-z_j)$ and $\sum_{i = 1}^n p_i z_i \leq \epsilon$ dropped and variable $z_j$ removed.
As $\xi^j \geq 0$ and $p_j \geq 0$, the intersection of the feasible sets of the two subproblems \eqref{T-leftbranch} and \eqref{T-rightbranch}, called \emph{overlap}, is  
\begin{equation}
	\label{Overlap}
	\CO=\left \{(x,v,z) \in\CC \stt v \geq \xi^{j},
	~p_{j}+\sum_{i \in [n]\backslash\{{j}\}} p_i z_i \leq \epsilon \right\}.
\end{equation}
The {overlap} $\CO$ between the left and right branches is, however, generally nonempty.
As the \BnC algorithm explores this overlap in both branches, and the restrictions of both branches to this overlap provide the same objective value, there exists some redundancy in the search tree.
In particular, considerable efforts are likely to be spent in exploring many subsequent nodes of one branch, whose feasible sets are subsets of the overlap $\CO$,   but can be avoided as these nodes cannot provide a better solution than that of the other branch.
We use the following example to illustrate this weakness.
\begin{Example}
	\label{example1}
	Consider
	\begin{equation}
		\label{expro1}
		\min \left\{6x_1 + x_2 + 3x_3 \stt \mathbb{P} \{x \geq \xi\} \geq 1- \epsilon,~x \in \R_+^3 \right\},
	\end{equation}
	where $\epsilon = \frac{4}{7}$ and $\xi$ is a random vector with $7$ equi-probable scenarios:
	\begin{equation*}
		\xi^{1} = \left(\begin{array}{c}  2 \\ 1 \\ 12 \\ \end{array}\right),\,
		\xi^{2} = \left(\begin{array}{c}  3 \\ 1 \\ 10 \\ \end{array}\right),\,
		\xi^{3} = \left(\begin{array}{c}  4 \\ 2 \\ 7 \\ \end{array}\right),\,
	\end{equation*}
	\begin{equation*}
		\xi^{4} = \left(\begin{array}{c}  5 \\ 2 \\ 6 \\	\end{array}\right),\,
		\xi^{5} = \left(\begin{array}{c}  6 \\ 2 \\ 6 \\ \end{array}\right),\,
		\xi^{6} = \left(\begin{array}{c}  7 \\ 1 \\  4 \\ \end{array}\right),\,
		\xi^{7} = \left(\begin{array}{c}  12 \\ 1 \\ 2 \\ \end{array}\right).
	\end{equation*}
	In formulation \eqref{milp} of this example, we have $x_1 = v_1$, $x_2 = v_2$, and $x_3 = v_3$.
	As a result, formulation \eqref{milp} reduces to
	\begin{equation}
		\small
		\label{expro}
		\!\min \left\{6v_1 + v_2 + 3v_3 \stt v \geq \xi^i (1-z_i),~\forall~i \in [7],~\frac{1}{7}\sum_{i =1}^7 z_i \leq \frac{4}{7},~v \in \R_+^3,~z \in \{0,1\}^7 \right\}.
	\end{equation}
	The optimal value of problem \eqref{expro} is $59$.
	We apply the {\rm \BnC} algorithm to solve problem \eqref{expro}.
	We assume that a feasible solution of objective value $59$ is found at the root node {\rm(}e.g., by some heuristic algorithm{\rm)}. 
	For simplicity of illustration, we use the {most infeasible branching} rule {\rm \cite{Achterberg2007}} to choose the variable to branch on.
	The {\rm \BnC} search tree is drawn in \cref{figure1}.
	At each node, we report the optimal value $z_{\rm\LP}$ of its {\LP} relaxation.
	
	The feasible regions of nodes $2$ and $3$ {\rm(}after removing the common fixed variable $z_4${\rm)} are 
	\begin{equation*}
		\begin{aligned}
			& \F^2 = \left\{(v,z)\in\R_+^3\times\{0,1\}^6 \stt  ~\frac{1}{7}(z_1+z_2+z_3+z_5+z_6 +z_7) \leq \frac{4}{7},\right.\\
			& \qquad \qquad\qquad \qquad\quad \left.~v\geq \xi^4 = (5,\,2,\,6)^\T,~v\geq\xi^i(1-z_i),~\forall~i\in\{1,2,3,5,6,7\} \vphantom{\frac{1}{7}}\right\},\\
			& \F^3 = \left\{(v,z)\in\R_+^3\times\{0,1\}^6 \stt  ~\frac{1}{7}(z_1+z_2+z_3+z_5+z_6 +z_7) \leq \frac{3}{7},\right.\\
			& \qquad \qquad\qquad \qquad\qquad \qquad\qquad\quad\left.~v\geq \boldsymbol{0},~v\geq\xi^i(1-z_i),~\forall~i\in\{1,2,3,5,6,7\} \vphantom{\frac{1}{7}} \right\},\\
		\end{aligned}
	\end{equation*}
	and the overlap is 
	\begin{equation*}
		\begin{aligned}
			& \CO= \F^2 \cap \F^3 = \left\{(v,z)\in\R_+^3\times\{0,1\}^6 \stt \frac{1}{7}(z_1+z_2+z_3+z_5+z_6+z_7) \leq \frac{3}{7},\right.\\
			& \left.\qquad \qquad\qquad \qquad\qquad v\geq \xi^4= (5,\,2,\,6)^\T,\,v\geq\xi^i(1-z_i),~\forall~i\in\{1,2,3,5,6,7\}\vphantom{\frac{1}{7}}\right\}.
		\end{aligned}
	\end{equation*}
	Consider node $6$,  a subsequent node of node $3$, whose feasible region {\rm(}after removing the fixed variable $z_4${\rm)} is 
	\begin{equation*}
		\begin{aligned}
			& \F^6 = \left\{(v,z)\in\R_+^3\times\{0,1\}^6 \stt  z_5 = 0,~\frac{1}{7}(z_1+z_2+z_3+z_5+z_6+z_7) \leq \frac{3}{7},\right.\\
			& \qquad \qquad\qquad\qquad\qquad \left. v\geq \xi^5=(6,\,2,\,6)^\T,\,v\geq\xi^i(1-z_i),~\forall~i\in\{1,2,3,5,6,7\} \vphantom{\frac{1}{7}}\right\}.
		\end{aligned}
	\end{equation*}
	Clearly, $\F^6 \subseteq \CO \subseteq \F^2$, implying that the optimal value of node $6$ cannot be better than that of node $2$.
	Thus, node $6$ {\rm{(}}and its descendant nodes $12$, $13$, $20$, $21$, $30$, and $31${\rm{)}} can be pruned.
	
	This example demonstrates the weakness of applying the {\rm \BnC} algorithm with the classic variable branching to solve formulation \eqref{milp}: a wasteful duplication of  efforts is spent in exploring the uninteresting overlap in the search tree.
	
	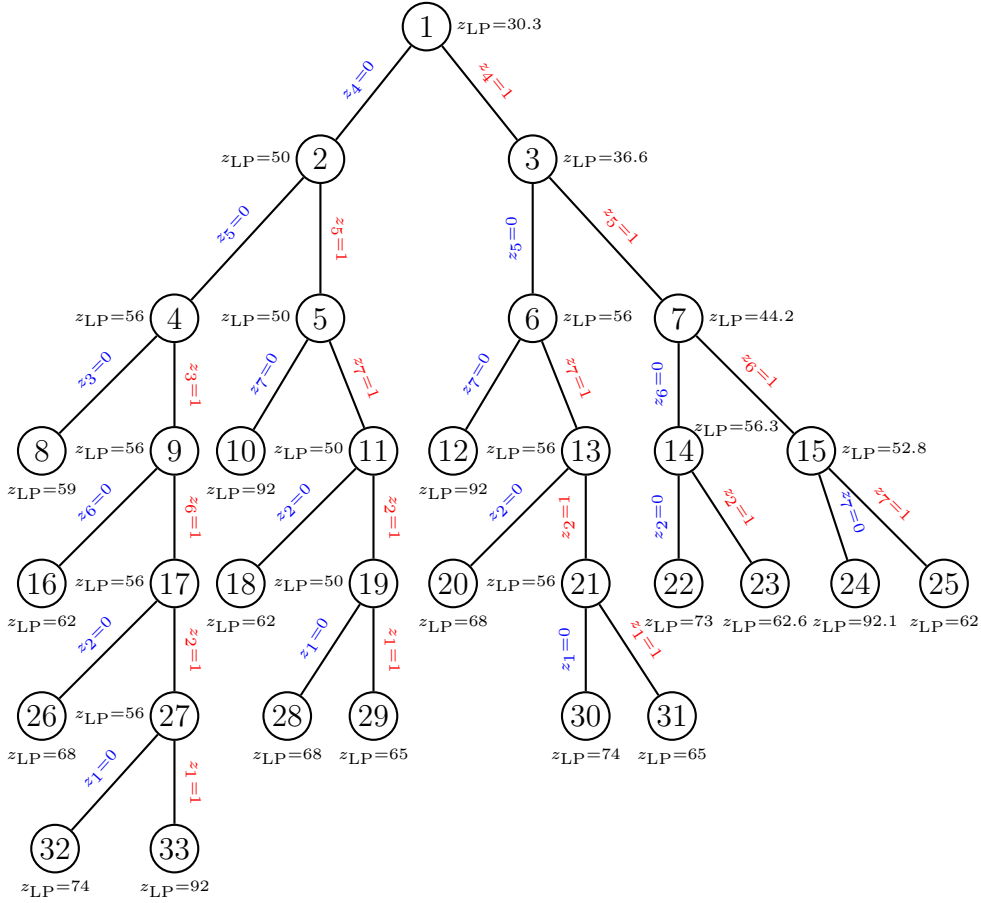
\begin{figure}[!h]
	\everymath{\scriptscriptstyle}
	\centering
	\begin{tikzpicture}[=mid,scale=0.5,thick]
		
		\node[vertex] (1) at (380bp,1000bp) {1};
		\node[vertex] (2) at (300bp,900bp) {2};
		\node[vertex] (3) at (460bp,900bp) {3};
		\node[vertex] (4) at (190bp,780bp) {4};
		\node[vertex] (5) at (300bp,780bp) {5};
		\node[vertex] (6) at (459.99bp,780bp) {6};
		\node[vertex] (7) at (570bp,780bp) {7};
		\node[vertex] (8) at (90bp, 680bp) {8};
		\node[vertex] (9) at (190bp, 680bp) {9};
		\node[vertex] (10) at (240bp,  680bp) {10};
		\node[vertex] (11) at (340bp,  680bp) {11};
		\node[vertex] (12) at (400bp, 680bp) {12};
		\node[vertex] (13) at (500bp,680bp) {13};
		\node[vertex] (14) at (569.999bp, 680bp) {14};
		\node[vertex] (15) at (670bp, 680bp) {15};
		\node[vertex] (16) at (90bp, 580bp) {16};
		\node[vertex] (17) at (190bp,580bp) {17};
		\node[vertex] (18) at (240bp,580bp) {18};
		\node[vertex] (19) at (340bp,580bp) {19};
		\node[vertex] (20) at (400bp,580bp) {20};
		\node[vertex] (21) at (499.999bp,580bp) {21};
		\node[vertex] (22) at (569.998bp,580bp) {22};
		\node[vertex] (23) at (635bp,580bp) {23};
		\node[vertex] (24) at (703bp,580bp) {24};
		\node[vertex] (25) at (770bp,580bp) {25};
		\node[vertex] (26) at (90bp, 480bp) {26};
		\node[vertex] (27) at (190bp, 480bp) {27};
		\node[vertex] (28) at (275bp,480bp) {28};
		\node[vertex] (29) at (340bp,480bp) {29};
		\node[vertex] (30) at (499.998bp,480bp) {30};
		\node[vertex] (31) at (565bp,480bp) {31};
		\node[vertex] (32) at (100bp, 380bp) {32};
		\node[vertex] (33) at (190bp, 380bp) {33};
		
		\draw (1) -- (2) node [midway,above,sloped,blue] {$z_4=0$};
		\draw (1) -- (3) node [midway,above,sloped,red] {$z_4=1$};
		\draw (2) --(4) node[midway,above,sloped,blue] {$z_5=0$};
		\draw (2) --(5) node[midway,above,sloped,red]  {$z_5=1$};
		\draw (3)--(6) node[midway,above,sloped,blue] {$z_5=0$};
		\draw (3)--(7) node[midway,above,sloped,red] {$z_5=1$};
		\draw (4)--(8) node[midway,above,sloped,blue] {$z_3=0$};
		\draw (4)--(9)  node[midway,above,sloped,red] {$z_3=1$};
		\draw (5)--(10)  node[midway,above,sloped,blue] {$z_7=0$};
		\draw (5)--(11) node[midway,above,sloped,red] {$z_7=1$};
		\draw (6)--(12) node[midway,above,sloped,blue] {$z_7=0$};
		\draw (6)--(13) node[midway,above,sloped,red] {$z_7=1$};
		\draw (7)--(14)  node[midway,above,sloped,blue] {$z_6=0$};
		\draw (7)--(15)  node[midway,above,sloped,red] {$z_6=1$};
		\draw (9)--(16) node[midway,above,sloped,blue] {$z_6=0$};
		\draw (9)--(17) node[midway,above,sloped,red] {$z_6=1$};
		\draw (11)--(18) node[midway,above,sloped,blue] {$z_2=0$};
		\draw (11)--(19) node[midway,above,sloped,red] {$z_2=1$};
		\draw (13)--(20) node[midway,above,sloped,blue] {$z_2=0$};
		\draw (13)--(21)  node[midway,above,sloped,red] {$z_2=1$};
		\draw (14)--(22) node[midway,above,sloped,blue]  {$z_2=0$};
		\draw (14)--(23) node[midway,above,sloped,red] {$z_2=1$};
		\draw (15)--(24) node[midway,above,sloped,blue]  {$z_7=0$};
		\draw (15)--(25) node[midway,above,sloped,red] {$z_7=1$};
		\draw (17)--(26) node[midway,above,sloped,blue] {$z_2=0$};
		\draw (17)--(27) node[midway,above,sloped,red] {$z_2=1$};
		\draw (19)--(28) node[midway,above,sloped,blue] {$z_1=0$};
		\draw (19)--(29)  node[midway,above,sloped,red] {$z_1=1$};
		\draw (21)--(30) node[midway,above,sloped,blue]  {$z_1=0$};
		\draw (21)--(31) node[midway,above,sloped,red] {$z_1=1$};
		\draw (27)--(32) node[midway,above,sloped,blue] {$z_1=0$};
		\draw (27)--(33) node[midway,above,sloped,red] {$z_1=1$};
		\begin{scope}[nodes = {left=7pt}]
			\node at (2) {{$z_{\text{LP}}=50$}};
			\node at (4) {{$z_{\text{LP}}=56$}};
			\node at (5) {{$z_{\text{LP}}=50$}};
			\node at (9) {{$z_{\text{LP}}=56$}};
			\node at (11) {{$z_{\text{LP}}=50$}};
			\node at (13) {{$z_{\text{LP}}=56$}};
			\node at (17) {{$z_{\text{LP}}=56$}};
			\node at (19) {{$z_{\text{LP}}=50$}};
			\node at (21) {{$z_{\text{LP}}=56$}};
			\node at (27) {{$z_{\text{LP}}=56$}};
		\end{scope}
		\begin{scope}[nodes = {right = 7pt, shift={(-0.2cm,0.3cm)}}]
			\node at (14) {{$z_{\text{LP}}=56.3$}};
		\end{scope}
		\begin{scope}[nodes = {right=7pt}]
			\node at (1) {{$z_{\text{LP}}=30.3$}};
			\node at (3) {{$z_{\text{LP}}=36.6$}};
			\node at (6) {{$z_{\text{LP}}=56$}};
			\node at (7) {{$z_{\text{LP}}=44.2$}};
			\node at (15) {{$z_{\text{LP}}=52.8$}};
		\end{scope}
		\begin{scope}[nodes = {below = 8pt}]
			\node at (8) {{$z_{\text{LP}}=59$}};
			\node at (10) {{$z_{\text{LP}}=92$}};
			\node at (12) {{$z_{\text{LP}}=92$}};
			\node at (16) {{$z_{\text{LP}}=62$}};
			\node at (18) {{$z_{\text{LP}}=62$}};
			\node at (20) {{$z_{\text{LP}}=68$}};
			\node at (22) {{$z_{\text{LP}}=73$}};
			\node at (23) {{$z_{\text{LP}}=62.6$}};
			\node at (24) {{$z_{\text{LP}}=92.1$}};
			\node at (25) {{$z_{\text{LP}}=62$}};
			\node at (26) {{$z_{\text{LP}}=68$}};
			\node at (28) {{$z_{\text{LP}}=68$}};
			\node at (29) {{$z_{\text{LP}}=65$}};
			\node at (30) {{$z_{\text{LP}}=74$}};
			\node at (31) {{$z_{\text{LP}}=65$}};
			\node at (32) {{$z_{\text{LP}}=74$}};
			\node at (33) {{$z_{\text{LP}}=92$}};
		\end{scope}
	\end{tikzpicture}
	\caption{The \BnC search tree of the problem in \cref{example1} with the classic variable branching applied.}
	\label{figure1}
\end{figure}
\end{Example}

\subsection{Removing the overlaps}
\label{subsect:removing overlaps}
To remove the overlap $\CO$ in the search tree, let us first divide the right branch \eqref{T-rightbranch} into the following two subproblems:
\begin{equation}\label{rightbranch1}
	O^{\text{R}_1} = \min\left\{c^\T x \stt (x,v,z) \in \CC, ~v \geq \xi^j,
	~p_j+\sum_{i \in [n]\backslash\{j\}} p_i z_i \leq \epsilon \right\},
\end{equation}
and 
\begin{equation}\label{rightbranch-new}
	O^{\text{R}_2} = \min\left\{c^\T x \stt (x,v,z) \in \CC, ~v \ngeq \xi^j, 
	~p_j+\sum_{i \in [n]\backslash\{j\}} p_i z_i \leq \epsilon \right\}.
\end{equation}
Here $v \ngeq \xi^j$ denotes that $v < \xi^j_{k}$ holds for at least one $k \in [m]$.
Observe that the feasible region of subproblem \eqref{rightbranch1} is identical to the overlap $\mathcal{O}$ in \eqref{Overlap} and thus a subset of the feasible region of problem \eqref{T-leftbranch}.
This, together with the fact that the objective functions of problems \eqref{T-leftbranch} and \eqref{rightbranch1} are identical, implies $O^{\text{R}_1} \geq O^{\text{L}}$.
Consequently, instead of exploring problem \eqref{T-rightbranch} with a potentially large feasible region, 
we can explore its restriction \eqref{rightbranch-new} in the search tree.
Notice that since constraints $v = Tx \geq \xi^j$ and $v=Tx \ngeq \xi^j$ appear in the left and (new) right branches \eqref{T-leftbranch} and \eqref{rightbranch-new}, respectively, no overlap exists between the feasible regions (and their projections onto the $x$ space) of the two branches.
\begin{Remark}
	\label{remark:left}
	By dividing the left branch \eqref{T-leftbranch} into 
	\begin{align}
		& O^{\rm{L}_1}=\min\left\{ c^\T x \stt (x,v,z) \in \CC,~v \geq \xi^{j},
		~ \sum_{i \in [n]\backslash\{j\}} p_i z_i \leq \epsilon - p_j \right\},\label{leftbranch1}\\
		& O^{\rm{L}_2}=\min\left\{ c^\T x \stt (x,v,z) \in \CC,~v \geq \xi^{j},
		~ \epsilon - p_j< \sum_{i \in [n]\backslash\{j\}} p_i z_i \leq \epsilon \right\},\label{leftbranch-new}
	\end{align}
	and noting that $O^{\rm{L}_1} \geq O^{\rm{R}}$, we can also keep the right branch \eqref{T-rightbranch} but add the new left branch \eqref{leftbranch-new} with a smaller feasible region than  that of \eqref{T-leftbranch} into the search tree. 
	Clearly, the feasible regions of problems \eqref{leftbranch-new} and \eqref{T-rightbranch} also do not contain an overlap.
	However, different from those of branches \eqref{T-leftbranch} and \eqref{rightbranch-new}, 
	the projections of the feasible regions of branches \eqref{leftbranch-new} and \eqref{T-rightbranch} onto the $x$ space may still contain overlaps.
	An illustrative example for this is provided in Appendix \ref{sec:example}.
	Thus, the redundancy still exists in the search tree as an optimal solution $x^*$ of \eqref{ccps} may simultaneously define feasible solutions of the two branches \eqref{leftbranch-new} and \eqref{T-rightbranch}.
\end{Remark}

\begin{Remark}
	The overlap in the {\rm \BnC} search tree has also been identified by 
	{\rm \citet{Qiu2014}} and {\rm \citet{Chen2023}} in the context of solving the $k$-violation
	linear programming and covering location problems. 
	\rev{In particular, such problems contain indicator constraints of the form $z_i = 0$ $\Rightarrow$ $A_i^\top x \geq b_i$,
	and when  branching on binary variable $z_i$, the linear constraint  $A_i^\top x \leq b_i$ {\rm(}or $A_i^\top x \leq b_i-1$ if $A_i \in \mathbb{Z}^d$ and $b_i \in \mathbb{Z}${\rm)} can be enforced in the right branch $(z_i =1)$ to remove the overlap.}
	However, unlike those in 
	{\rm\cite{Qiu2014}} and {\rm\cite{Chen2023}} where the overlap can be removed by linear constraints, we need to use the highly nonlinear constraints  $v \ngeq \xi^j$ to remove the overlap arising in solving the \CCPs.
\end{Remark}

The overlap $\CO$ in \eqref{Overlap} can also be removed by refining formulation \eqref{milp}.
To do this, let us first present problem \eqref{ccps} as
\begin{equation}
	\label{ccps2}\tag{CCP'}
	\min\left\{c^\T x \stt Tx=v, ~x \in \X, ~\sum_{i=1}^n p_i \chi (v \ngeq \xi^i) \leq \epsilon \right\},
\end{equation}
where $\chi$ is an indicator function:
\begin{equation}
	\label{indicator2}
	\chi(v \ngeq \xi^i) = \left\{\begin{array}{ll}
		1, & \ \text{if}~v \ngeq \xi^i;\\[3pt]
		0, & \ \text{otherwise},
	\end{array}\right.~\forall~i \in[n].
\end{equation}
Let $\F^{\text{\CCP}}$ and $\F^{\text{\CCP'}}$ denote the feasible regions of problems \eqref{ccps} and \eqref{ccps2}, respectively.
Then we must have $\F^{\text{\CCP}} = \proj_x(\F^{\text{\CCP'}})$, and hence problems \eqref{ccps} and \eqref{ccps2} are equivalent.
Unlike formulation \eqref{ccps2} where  \eqref{indicator2}  simultaneously ensures $\chi(v \ngeq \xi^i) =0 \Rightarrow v \geq \xi^i$ and $\chi(v \ngeq \xi^i) =1 \Rightarrow v \ngeq \xi^i$, formulation \eqref{milp} with \eqref{milp-consvz} and $z_{i}\in \{0,1\}$ only guarantees $z_{i}=0 \Rightarrow v \geq \xi^i$ but cannot guarantee $z_{i}=1 \Rightarrow v \ngeq \xi^i$. 
In other words, it is possible that formulation \eqref{milp} has a feasible solution $(x^*,v^*,z^*)$ such that $z^*_{i}=1$ and $v^* \geq \xi^i$ (though setting $z^*_{i}=0$ yields another feasible solution of formulation \eqref{milp} with the same objective value and hence formulation \eqref{milp} is a valid formulation for \eqref{ccps}).

However, only enforcing $z_{i}=0 \Rightarrow v \geq \xi^i$ in formulation \eqref{milp} makes the overlaps in the search tree.
To see this, let us add the valid nonlinear \emph{if-then} constraints
\begin{equation}
	\label{cons-ifthen}
	z_i = 1 \Rightarrow v \ngeq \xi^i, ~\forall~ i \in [n]\,,\\
\end{equation}
into formulation \eqref{milp} and obtain a mixed integer nonlinear programming (\MINLP) formulation for \eqref{ccps}:
\begin{equation}
	\label{minlp}\tag{MINLP}
	\min\left\{c^\T x \stt \eqref{milp-consx}-\eqref{milp-consbound},\eqref{cons-ifthen} \right\}.
\end{equation}
Let $\F^{\text{\MILP}}$ and $\F^{\text{\MINLP}}$ be the feasible sets of formulations \eqref{milp} and \eqref{minlp}, respectively.
As $\proj_x(\F^{\text{\MILP}})= \proj_x(\F^{\text{\MINLP}})=\F^{\text{\CCP}}$,
formulations \eqref{minlp} and \eqref{milp} are equivalent.
Now, for formulation \eqref{minlp}, branching on variable $z_j$ will not lead to an overlap in the $x$ space as $v=Tx\geq \xi^j$ and $v=Tx \ngeq \xi^j$ are enforced in the left and right branches, respectively.
As constraints \eqref{milp-consvz} and \eqref{cons-ifthen} are all enforced in problem \eqref{minlp},
a node $a=(\CN^a_0, \CN^a_1)$ in the search tree includes the linear constraints $v \geq \xi^i$, $i \in \CN_0^a$, and the \emph{disjunctive constraints} $v \ngeq \xi^i $, $i \in \CN_1^a$.
Therefore, the result can be generalized to any two descendant nodes of the left and right branches of a node, as detailed in the following theorem.

\begin{Theorem}
	\label{thm:overlap}
	Let $a=(\CN_0^a, \CN_1^a)$ and $b=(\CN_0^b, \CN_1^b)$ be any two nodes in the search tree of formulation \eqref{minlp} and $c$ be their first common ancestor. 
	If $c$ differs from $a$ and $b$, 
	then 
	\begin{equation}\label{projres}
		\proj_{x}(\F^{\text{\rm \MINLP}}(\CN_0^a, \CN_1^a))\cap 
		\proj_{x}(\F^{\text{\rm \MINLP}}(\CN_0^b, \CN_1^b))= \varnothing,
	\end{equation}
	where $\F^{\text{\rm \MINLP}}(\CN_0^a, \CN_1^a)$ and $\F^{\text{\rm \MINLP}}(\CN_0^b, \CN_1^b)$ are the feasible regions of nodes $a$ and $b$, respectively.
\end{Theorem}
\begin{proof}
	Without loss of generality, suppose that $z_j$ is branched on at node $c$, and nodes $a$ and $b$ are in the left and right branches of node $c$, respectively. 
	As $v=Tx \geq \xi^j$ and $v =Tx \ngeq \xi^j$ are enforced in the left and right branches, respectively, \eqref{projres} holds.
\end{proof}
By \cref{thm:overlap}, in the search tree of formulation \eqref{minlp}, no $x \in \R^d$ can simultaneously define feasible solutions for any two descendant nodes of the left and right branches of a node.
This is intrinsically different from the search tree of formulation \eqref{milp} in which a solution $x$ may define feasible solutions for two descendant nodes of the left and right branches of a node.
Thus, it can be expected that the search tree of applying the \BnC algorithm to formulation \eqref{minlp} is smaller than that of applying the \BnC algorithm to formulation \eqref{milp}.

In contrast to formulation \eqref{milp} which is an \MILP problem, formulation \eqref{minlp} is, however, a relatively hard \MINLP problem due to the nonlinear if-then constraints \eqref{cons-ifthen}.
Moreover, to the best of our knowledge, the highly nonlinear if-then constraints  \eqref{cons-ifthen} cannot be directly tackled by state-of-the-art \MINLP solvers.
In the next two sections, we will propose a novel branching and overlap-oriented node pruning and variable fixing techniques based on formulation \eqref{milp} 
that explicitly employ the if-then constraints \eqref{cons-ifthen} in removing the overlaps in the search tree.

\section{Dominance-based branching}
\label{sect:precedence branch}
As shown in \cref{subsect:removing overlaps}, while using \eqref{cons-ifthen} guarantees that no overlap exists between the left branch \eqref{T-leftbranch} and right branch \eqref{rightbranch-new}, it also leads to a hard \MINLP subproblem \eqref{rightbranch-new}.
To resolve the difficulty, in this section, we first leverage the concept of \emph{dominance} between scenarios \cite{Ruszczynski2002} and propose a novel branching, called dominance-based branching, which achieves a better tradeoff between the size of the overlap and the solution efficiency of the two branches. 
Then we discuss its relation to the classic variable branching that is applied to problem \eqref{milp} with the dominance inequalities in \cite{Ruszczynski2002}.
Subsequently, we enhance the proposed branching by presenting a technique to derive more dominance pairs between different scenarios.
Finally, we use an example to illustrate the effectiveness of the proposed dominance-based branching for solving problem \eqref{milp}.

\subsection{Description of the dominance-based branching}
\label{subsect:description}

A scenario $i$ is dominated by a scenario $j$ if $\xi^i \leq \xi^j$, denoted as $i \preceq j$.
Let
\begin{equation}
	\CN_j^-= \left\{i \in [n] \stt \xi^i \leq \xi^j \right\}
\end{equation}
be the index set of scenarios that are dominated by scenario $\xi^j$ (including scenario $\xi^j$).
For a feasible solution $(x,v,z)$ of problem \eqref{minlp} with $z_j =0$, we have $v \geq \xi^j$.
This, together with $\xi^i \leq \xi^j$, $z_i \in \{0,1\}$, and the if-then constraints $z_i=1 \Rightarrow v \ngeq \xi^i $ in \eqref{cons-ifthen} for $i \in  \CN_j^-$, implies $z_i=0$ for all $i \in \CN_j^-$ and  
 \begin{align}
		& \left\{(x,v,z) \stt \eqref{milp-consx}-\eqref{milp-consbound},~\eqref{cons-ifthen}, ~z_j =0 \right\} \label{Lb1}\\	
		& \qquad \subseteq  \left\{(x,v,z) \stt \eqref{milp-consx}-\eqref{milp-consbound},~z_i =0,~\forall~i \in \CN_j^- \right\} \label{Lb2} \\
		& \qquad \subseteq  \left\{(x,v,z) \stt \eqref{milp-consx}-\eqref{milp-consbound}, ~z_j=0\right\}.  \label{Lb3} 
\end{align}
Observe that \eqref{Lb1} and \eqref{Lb3} are the feasible sets of problems \eqref{minlp} and \eqref{milp} with $z_j =0$, respectively. 
By \eqref{Lb1}--\eqref{Lb3} and the fact that exploring a left branch with its feasible set being \eqref{Lb1} or \eqref{Lb3} in the search tree can all return a correct optimal solution for problem \eqref{ccps},
we can develop a new branching that explores a new \emph{linear} left branch with its feasible set being \eqref{Lb2}:
\begin{equation}
	\label{N-leftbranch}
	O^{\text{L}_\text{D}}=\min\left\{c^\T x \stt \eqref{milp-consx}-\eqref{milp-consbound},~z_i =0,~\forall~i \in \CN_j^- \right\}.
\end{equation}

Similarly, we can also derive a \emph{linear} right branch with a potentially more compact feasible set.
In particular, let
\begin{align}
	 \CN_j^+= \left\{i \in [n] \stt \xi^j \leq \xi^i \right\}
\end{align}
be the index set of scenarios that dominate scenario $\xi^j$  (including scenario $\xi^j$), and
$(x,v,z)$ be a feasible solution of problem \eqref{minlp} with $z_j =1$.
From \eqref{cons-ifthen}, $v \ngeq \xi^j$ holds, which, together with $\xi^j \leq \xi^i$ for $i \in  \CN_j^+$, implies that $v \ngeq \xi^i$ must hold. 
From $v \geq \xi^i (1-z_i)$ and $z_i \in \{0,1\}$, $z_i=1$ must be satisfied for $i \in \CN_j^+$.
Using a similar argument as that of deriving the left branch \eqref{N-leftbranch}, we can derive a new equivalent right branch:
\begin{equation}
	\label{N-rightbranch}
	O^{\text{R}_\text{D}}=\min\left\{c^\T x \stt \eqref{milp-consx}-\eqref{milp-consbound},~z_i =1,~\forall~i \in  \CN_j^+\right\}.
\end{equation}

Note that branches \eqref{N-leftbranch} and \eqref{N-rightbranch} are of the form \eqref{milp}, and so the above dominance-based branching  can be applied to any node of the search tree. 
In particular, for a node $(\CB_0, \CB_1)$ in the search tree, the index sets of variables, that can be  fixed at zero and one, are 
\begin{equation}
	\label{N0N1}
	\begin{aligned}
	\CN_0 = \bigcup_{j \in \CB_0}  \CN_j^-~\text{and}~
	\CN_1 = \bigcup_{j \in \CB_1}  \CN_j^+,
	\end{aligned}
\end{equation}
respectively.

Compared with the two branches \eqref{T-leftbranch} and \eqref{T-rightbranch} with the classic variable branching applied, the new branches \eqref{N-leftbranch} and \eqref{N-rightbranch} with the proposed dominance-based branching applied are still \MILP problems but with more compact feasible regions.
Thus, compared with the overlap $\mathcal{O}$ in \eqref{Overlap}, the overlap between the feasible regions of branches \eqref{N-leftbranch} and \eqref{N-rightbranch} (after removing the common fixed variable $z_j$), i.e., 
{\small
\begin{equation*}
	\begin{aligned}
	&\CO'=\left \{(x,v,z) \in \CC \stt v \geq \xi^{j},
	~p_{j}+\sum_{i \in [n]\backslash\{{j}\}} p_i z_i \leq \epsilon, \right.\\
	& \left.\qquad \qquad\qquad~z_i =0,~\forall~i \in  \CN_j^-\backslash\{j\}, ~z_i =1,~\forall~i \in  \CN_j^+\backslash\{j\} \vphantom{\sum_{i \in [n]\backslash\{{j}\}} p_i z_i} \right\},
	\end{aligned}
\end{equation*}}%
is much smaller, especially when $ \CN_j^-$ and  $ \CN_j^+$ are large.
Consequently, certain nodes that are explored with the classic variable branching applied may not be explored with the proposed dominance-based branching applied (i.e., the feasible sets of these nodes are subsets of $\CO \backslash \CO'$). 
Due to this advantage, it can be expected that the \BnC search tree with the proposed dominance-based branching applied will be potentially much smaller than that with the classic variable branching applied.

The more compact feasible region of the new right branch \eqref{N-rightbranch} may lead to infeasibility or enable to provide a potentially much stronger \LP relaxation bound than that provided by the old one \eqref{T-rightbranch}. 
An example for illustrating this will be provided in \cref{subsect:example}.
This advantage opens up more possibilities to apply the infeasibility or bound exceeding based node pruning and guides the selection of branching variables (if strong branching or its variants \cite{Achterberg2005} are applied), thereby enabling to further reduce the tree size.
Notice that for an optimal solution $(x,v,z)$ of the \LP relaxation of branch \eqref{T-leftbranch}, setting $z_{i}=0$ for all $i \in  \CN_j^-$ yields a feasible solution of branch \eqref{N-leftbranch} with the same objective value, and thus the new left branch \eqref{N-leftbranch} provides the same \LP relaxation bound as that provided by the old one \eqref{T-leftbranch}.

It is worth remarking that the proposed dominance-based branching can be implemented along with various \emph{branching strategies} such as strong branching and the most infeasible branching \cite{Achterberg2007} that choose a variable to branch on for the current \LP relaxation.
In \cref{sect:computational results}, we will describe how to implement the proposed dominance-based branching into a state-of-the-art open source \MILP solver while using its default fine-tuned branching strategy.

\subsection{Relation to the result in \cite{Ruszczynski2002}}
\label{subsect:existing work}

As shown in the previous subsection, for a feasible solution $(x,v,z)$ of problem \eqref{minlp}, if $z_j=0$, then $z_i =0$ holds for $i \in [n]\backslash \{j\}$ with $\xi^i \leq \xi^j$ (or equivalently, 
if $z_i=1$, then $z_j =1$ holds for $j \in [n]\backslash \{i\}$ with $ \xi^i \leq \xi^j$).
Thus, the following \emph{dominance inequalities} 
\begin{equation}
	\label{precedenceconstraints}
	z_i \leq z_j, ~\forall~(i, j) \in \A := \{(i, j) \in [n]\times [n] \stt \xi^{i} \leq \xi^{j}, ~i \neq j\}
\end{equation}
are valid for problem \eqref{minlp}.
This, together with the equivalence of problems \eqref{minlp} and \eqref{milp}, shows that inequalities \eqref{precedenceconstraints} are also valid for problem \eqref{milp}. 
This result, first established by \citet{Ruszczynski2002}, is formally stated as follows.
\begin{Theorem}[\cite{Ruszczynski2002}]
	\label{thm:sameobj}
	The \MILP problem
	\begin{equation}
		\label{RF}
		\min\left\{c^\T x \stt \eqref{milp-consx}-\eqref{milp-consbound}, \eqref{precedenceconstraints} \right\}
	\end{equation}
	is equivalent to problem \eqref{milp} in terms of sharing at least one identical optimal solution.
\end{Theorem}

\begin{Remark}
	\label{compactineq}
	If $\xi^i \leq \xi^s \leq \xi^j$ for some distinct $i,j,s \in [n]$, the dominance inequality $z_i \leq z_j$ is implied by inequalities $z_i \leq z_s$ and $z_s \leq z_j$.
	Thus, constraints \eqref{precedenceconstraints} in problem \eqref{RF} can be simplified as $z_i\leq z_j$ for $(i,j) \in \A'$ where $$\A' :=\left\{(i, j) \in \A \stt \text{no $s \in [n]\backslash \{i,j\}$ with $\xi^i \leq \xi^s \leq \xi^j$ exists}\right\}.$$
\end{Remark}

We can further strengthen the result in \cref{thm:sameobj} by showing that the {\rm \LP} relaxations of problems \eqref{milp} and \eqref{RF} are equivalent. 
This indicates that the dominance inequalities cannot improve the \LP relaxation of problem \eqref{milp}. 
\begin{Proposition}
	\label{lemma:samelprelaxation}
	The {\LP} relaxations of problems \eqref{milp} and \eqref{RF} are equivalent. 
\end{Proposition}
\begin{proof}	
	It suffices to show there exists an optimal solution $(x^*, v^*, z^*)$ of the \LP relaxation of problem \eqref{milp} 
	such that ${z}^*_{i} \leq {z}^*_{j}$ holds for all $(i,j) \in \A$. 
	Let $(\bar{x}, \bar{v}, \bar{z})$ be an optimal solution of the \LP relaxation of problem \eqref{milp}.
	If $\bar{z}_{i} \leq \bar{z}_{j}$ holds for all $(i,j)\in \A$, the statement follows.
	Otherwise, there exists $(i_0, j_0) \in \A$ such that $\bar{z}_{i_0} > \bar{z}_{j_0}$. 
	In this case, we construct a new point $(\hat{x}, \hat{v}, \hat{z})$ by setting 
	\begin{equation}
		\label{hatxdef}
		\hat{x} = \bar{x}, ~\hat{v} = \bar{v}, ~\hat{z}_{i_{0}} = \bar{z}_{j_{0}},
		~\text{and} ~\hat{z}_i = \bar{z}_i, ~\text{for all} ~i \in [n]\backslash \{i_{0}\}.
	\end{equation}
	Clearly, the objective values of the \LP relaxation of problem
	\eqref{milp} at points $(\bar{x}, \bar{v}, \bar{z})$ and $(\hat{x}, \hat{v}, \hat{z})$, respectively, are the same.
	In addition, $(\hat{x}, \hat{v}, \hat{z})$ is also a feasible solution of the \LP relaxation of problem \eqref{milp} as
	\begin{equation*}
		\hat{v} = \bar{v} \geq \xi^{j_{0}}(1-\bar{z}_{j_{0}})
		\geq \xi^{i_{0}}(1-\bar{z}_{j_{0}}) = \xi^{i_{0}}(1-\hat{z}_{i_{0}}),
	\end{equation*}
	where the second inequality follows from $\xi^{i_0} \leq \xi^{j_0}$ and $\bar{z}_{j_0} \leq 1$.
	Recursively applying the above operation and using the fact that the value of each $z_i$ decreases at most $n$ times, we will obtain an optimal solution $(x^*, v^*, z^*)$ of the \LP relaxation of problem \eqref{milp} 
	such that ${z}^*_{i} \leq {z}^*_{j}$ holds for all $(i,j) \in \A$.
\end{proof}

Next, we show that (i)  applying the proposed dominance-based branching in \cref{subsect:description} to problem \eqref{milp}
is theoretically equivalent to (ii) applying the classic variable branching to problem \eqref{RF} in the following two aspects.
First, for two nodes in the two search trees constructed by (i) and (ii) characterizing by the same branching variables  $\CB_0$  and $\CB_1$, the index sets of variables fixed at zero and one are all identical.
Indeed, for case (i), as shown in \cref{subsect:description}, the index sets of variables fixed at zero and one are the $\CN_0$ and $\CN_1$, respectively, defined in \eqref{N0N1}; and
for case (ii), the index sets of variables fixed at zero and one at node $(\CB_0,\CB_1)$
(due to branching or fixing by the dominance inequalities in \eqref{precedenceconstraints})
are also $\CN_0$ and $\CN_1$, respectively.
Second, the \LP relaxations of the two nodes are also equivalent, as detailed in the following corollary.
\begin{Corollary}
	\label{thm:samenodelprelaxation1}
	Consider a node $(\CB_0, \CB_1)$ in the search tree constructed by {\rm (i)} or {\rm (ii)} and let $\CN_0$ and $\CN_1$ be defined as in \eqref{N0N1}.
	The {\LP} relaxation of problem \eqref{milp} with $z_{i}=0$ for $i \in \CN_0$ and $z_{i}=1$ for $i \in \CN_1$ is equivalent to the {\LP} relaxation of problem \eqref{RF} with $z_{i}=0$ for $i \in \CN_0$ and $z_{i}=1$ for $i \in \CN_1$.
\end{Corollary}
\begin{proof}
	The result follows from \cref{lemma:samelprelaxation} and the fact that  a subproblem of \eqref{milp} or \eqref{RF} with $z_i=0$ for $i \in \CN_0$ and $z_i=1$ for $i \in \CN_1$ still takes the form of \eqref{milp} or \eqref{RF}, respectively.
\end{proof}
\noindent 
\cref{thm:samenodelprelaxation1} implies that the \LP relaxations of the two nodes $(\CB_0, \CB_1)$ in the two search trees share at least one identical optimal solution (if they are feasible).
Thus, if we choose the same variable to branch on for the identical (fractional) \LP relaxation solution using, e.g., strong branching strategy \cite{Achterberg2007}, the search trees constructed by (i) and (ii) will also be identical.

The theoretical equivalence of (i) and (ii) shows that although the dominance inequalities in \eqref{precedenceconstraints} cannot strengthen the \LP relaxation of problem \eqref{milp}, they can, as the proposed dominance\rev{-based} branching, enhance the \BnC algorithm of \MILP solvers by removing the uninteresting overlap in the search tree. 
This theoretical equivalence also sheds useful insights on the proposed dominance-based branching.
More specifically, the dominance-based branching can be treated as an enhanced version of classic variable branching to formulation \eqref{milp} that additionally uses the dominance relations in \eqref{precedenceconstraints} \rev{as local cuts} for fixing variables \rev{at the nodes of the search tree. As such, it allows for the use of sophisticated \MILP methodologies such as cutting planes and preprocessing techniques.}
In particular,  effective techniques for \CCPs like mixing cuts \cite{Atamturk2000a,Gunluk2001b} and their variants \cite{Luedtke2010a,Kucukyavuz2012,Abdi2016,Zhao2017,Klnc-Karzan2022} and the preprocessing technique in \cite{Luedtke2010a}
can all be applied along with the proposed dominance-based branching. 

\rev{
\begin{Remark}
	The above theoretical interpretation of the dominance-based branching also allows us to extend it to the chance-constrained program with random technology matrix: $\min \{c^\top x : \mathbb{P}\{\tilde{T} x \geq h\} \geq 1-\epsilon, \, x \in \X \}$, 
	where $\tilde{T}$ is a random constraint matrix with 
	$\mathbb{P}\{\tilde{T} = T^i\} = p_i \geq 0$ for $i \in [n]$ and $\sum_{i=1}^n p_i = 1$.
	In particular, for the corresponding big-M formulation: 
	$\min \left\{c^\top x \,:\, T^i x \geq h - M_i z_i, \,\forall\,i \in [n], \,\sum_{i = 1}^n p_i z_i \leq \epsilon, \,x \in \X, \,z \in \{0,1\}^n\right\}$, 
	using the technique of \citet{Ruszczynski2002}, we can derive the dominance relations $z_i \leq z_j$ for some pairs $i, j \in [n]$ with $i\neq j$, and use these dominance relations to fix the variables at the nodes of the search tree $($as to remove the uninteresting overlap$)$.
\end{Remark}
}

Although applying the dominance-based branching to problem \eqref{milp}
is theoretically equivalent to applying the classic variable branching to problem \eqref{RF}, the former, however, can avoid solving a possibly large \LP relaxation of problem \eqref{RF} (due to the addition of the dominance inequalities in  \eqref{precedenceconstraints}).
Therefore, it can be expected that applying the dominance-based branching to problem \eqref{milp} is  more computationally efficient.
In \cref{sect:computational results}, we will further present computational results to illustrate this.

\subsection{\rev{Deriving more dominance relations}}
\label{subsect:preprocess}

\rev{The effectiveness of the dominance-based branching depends critically on the presence of the dominance pairs $i \preceq j$.}
In general, the more the dominance pairs, the smaller the overlap $\CO'$, and thus the more effective the dominance-based branching.
However, the condition $\xi^i \leq \xi^j $ \rev{for the dominance pair $i \preceq j$} is quite restrictive, and in some applications, the number of dominance pairs is extremely small, leading to the ineffectiveness of the dominance-based branching.
\rev{\citet{Dentcheva2000,Lejeune2008,Luedtke2010a,Lejeune2010} proposed a preprocessing technique to derive lower bounds for variables $v$, 
thereby obtaining an equivalent formulation with a smaller problem size or tighter \LP relaxation. 
In the following, we will show that using such lower bounds, we can derive more dominance pairs $i \preceq j$, 
and therefore further enhance the proposed dominance-based branching.}

To proceed, for $k \in [m]$, let $\{\pi_k(1),\pi_k(2),\ldots,\pi_k(n)\}$ be a permutation of $[n]$ such that $\xi^{\pi_k(1)}_k \geq \xi^{\pi_k(2)}_k \geq \cdots \geq \xi^{\pi_k(n)}_k$.
Define $\tau_k := \min\{s \stt \sum_{i = 1}^s p_{\pi_k(i)} > \epsilon\}$.
From the knapsack constraint \eqref{milp-consknap}, $z_{\pi_k(t)}=1$, $t=1,2, \ldots, \tau_k$, cannot simultaneously hold for a feasible solution $(x,v,z)$ of formulation \eqref{milp}. 
Using this observation and the fact that $v_k \geq \xi_k^{\pi_k(t)} (1-z_{\pi_{k}(t)})$ and $z_{\pi_{k}(t)}\in \{0,1\}$ for $t=1,2,\ldots,\tau_k$,  a lower bound $\xi^0_k := \xi^{\pi_k(\tau_k)}_k$ for variable $v_k$ can be derived.
\begin{Lemma}
	\label{lemma:defxi0}
	{\rm(\cite{Dentcheva2000,Lejeune2008,Luedtke2010a,Lejeune2010})}
	$v_k \geq \xi_k^0$, $k \in [m]$, are valid for formulation \eqref{milp}.
\end{Lemma}

Let
\begin{equation}
	\label{def-newxi}
	\bar{\xi}_k^{i} := \max\left\{\xi_k^{i},\,\xi_k^{0}\right\},~\forall~i \in [n],~k \in [m],
\end{equation}
and
\begin{equation}\label{milp-consvzeq}\tag{2'}
	v \geq \bar{\xi}^{i}(1-z_i),~\forall~i\in [n].
\end{equation}
From \cref{lemma:defxi0} and $z \in \{0,1\}^n$, \eqref{milp-consvz} and \eqref{milp-consvzeq} are equivalent.
Thus, applying this preprocessing technique, we can obtain the following new equivalent \MILP formulation for problem \eqref{ccps}:
\begin{equation}
	\label{newMILP}
	\min\left\{c^\T x \stt\eqref{milp-consx},~\eqref{milp-consvzeq},~\eqref{milp-consknap},~\eqref{milp-consbound} \right\}.
\end{equation}
In the new \MILP formulation \eqref{newMILP}, if $\bar{\xi}^i \leq \bar{\xi}^j$, scenario $i$ is dominated by scenario $j$.
Observe that from \eqref{def-newxi}, $\bar{\xi}^i \leq \bar{\xi}^j$ is more likely to \rev{appear} than ${\xi}^i \leq {\xi}^j$, and thus more dominance relations 
\begin{equation}
	\label{newprecedenceconstraints}
	z_i \leq z_j, ~\forall~(i, j) \in \bar{\A} := \left\{(i, j) \in [n]\times [n] \stt\bar{\xi}^i \leq \bar{\xi}^j, ~i \neq j\right\}
\end{equation}
can be derived.
Now, applying the dominance-based branching with the dominance relations in $\bar{\A}$, we will obtain two more compact branches than \eqref{N-leftbranch} and \eqref{N-rightbranch}:
\begin{align}
		& \min\left\{c^\T x \stt \eqref{milp-consx}-\eqref{milp-consbound},~z_i =0,~\forall~i \in  \bar{\CN}_j^- \right\},	\label{strongleftbranch} \\
	& \min\left\{c^\T x \stt \eqref{milp-consx}-\eqref{milp-consbound},~z_i =1,~\forall~i \in \bar{\CN}_j^+\right\},	\label{strongrightbranch}
\end{align}
where $\bar{\CN}_j^- := \left\{i \in [n] \stt \bar{\xi}^i \leq \bar{\xi}^j \right\}$ 
and $\bar{\CN}_j^+ := \left\{i \in [n] \stt \bar{\xi}^j \leq \bar{\xi}^i \right\}$.
Moreover, with the decreasing of $\epsilon$, $\xi_k^0$ will become larger and more the dominance pairs $i \preceq j$ are likely to appear, implying that branches \eqref{strongleftbranch} and \eqref{strongrightbranch} will also become much more compact.

\begin{Remark}\label{re:strengthenrf}
	Adding 
	\begin{equation}
		\label{newprecedenceconstraints1}
		\!\!\! z_i \leq z_j, ~\forall~(i, j) \in \bar{\A}' := \left\{(i, j) \in \bar{\A}\, : \,  \text{no $s \in [n]\backslash \{i,j\}$ with $\bar{\xi}^i \leq \bar{\xi}^s \leq \bar{\xi}^j$ exists}\right\},
	\end{equation}
	the simplified version of \eqref{newprecedenceconstraints} {\rm(}see \cref{compactineq}{\rm)},
	into formulation \eqref{newMILP} yields another equivalent {\MILP} formulation. 
	This formulation is stronger than \eqref{RF} in terms of providing a more compact feasible region.
\end{Remark}

\subsection{An illustrative example}
\label{subsect:example}
We now apply the dominance-based branching to problem \eqref{expro} in \cref{example1} to demonstrate its effectiveness over the classic variable branching.

We first note that for problem \eqref{expro}, the lower bounds $\xi^0$ for variables $v$ stated in \cref{lemma:defxi0} \rev{read} $\left(\begin{array}{c} 4 \\ 1 \\ 6 \\ \end{array}\right)$.
Applying the preprocessing technique in \cref{subsect:preprocess}, we will obtain an equivalent  problem of \eqref{expro} where  $\xi^i$ is replaced by $\bar{\xi}^i$:
\begin{equation*}
	\bar{\xi}^{1} = \left(\begin{array}{c}  4 \\ 1 \\ 12 \\ \end{array}\right),\,
	\bar{\xi}^{2} = \left(\begin{array}{c}  4 \\ 1 \\ 10 \\ \end{array}\right),\,
	\bar{\xi}^{3} = \left(\begin{array}{c}  4 \\ 2 \\ 7 \\ \end{array}\right),\,
\end{equation*}
\begin{equation*}
	\bar{\xi}^{4} = \left(\begin{array}{c}  5 \\ 2 \\ 6 \\	\end{array}\right),\,
	\bar{\xi}^{5} = \left(\begin{array}{c}  6 \\ 2 \\ 6 \\ \end{array}\right),\,
	\bar{\xi}^{6} = \left(\begin{array}{c}  7 \\ 1  \\ 6 \\ \end{array}\right),\,
	\bar{\xi}^{7} = \left(\begin{array}{c}  12 \\ 1 \\ 6 \\ \end{array}\right).
\end{equation*}
With the preprocessing technique, we can detect $3$ dominance pairs: $2 \preceq 1$, $6\preceq 7$, and $4 \preceq 5$.
In contrast, only a single dominance pair $4 \preceq 5$ can be detected without the preprocessing technique.  

\cref{figure-newbranch} displays the search tree constructed by using the dominance-based branching to solve problem \eqref{expro} with the most infeasible branching rule applied.
Let us consider the branching of variable $z_4$ at the root node. 
Since $4 \preceq 5$,  we obtain  $\CN_4^+ = \{4,5\}$. 
The right branch $3$ of node $1$ is associated with $\CN_0 = \varnothing$ and $\CN_1 = \{4,5\}$.
Thus, node $6$ and its descendant nodes $12$, $13$, $20$, $21$, $30$, and $31$ in the previous search tree in \cref{figure1} do not need to be explored with the proposed dominance-based branching applied.

Using the classic variable branching, the \LP relaxation of the right branch $27$ of node $17$ in \cref{figure1} is still feasible with an optimal value of $56$, while using the proposed dominance-based branching, the \LP relaxation of the right branch $27$ of node $17$ in \cref{figure-newbranch} is infeasible, thereby avoiding further branching at node $27$.
Similarly, using the proposed dominance-based branching, the \LP relaxation of the right branch $19$ of node $11$ in \cref{figure-newbranch} has an optimal value of $65$, which is larger than the optimal value of problem \eqref{expro} (i.e.,  $59$), thereby also avoiding further branching at node $19$.

This example shows the effectiveness of the proposed dominance-based branching over the classic variable branching in reducing the search tree size.
Overall, using the proposed dominance-based branching with the preprocessing technique,  
only $19$ nodes need to be explored, which is $14$ less than that of the search tree in \cref{figure1} where the classic variable branching is applied.
\begin{figure}[!h]
	\everymath{\scriptscriptstyle}
	\centering
	\begin{tikzpicture}
		[level distance=13mm, thick,
		level 1/.style={sibling distance=66mm},
		level 2/.style={sibling distance=28mm},
		level 3/.style={sibling distance=26mm},
		level 4/.style={sibling distance=24mm},
		level 5/.style={sibling distance=22mm}]
		\node[vertex](1) {1}
		child{node[vertex](2) {2}
			child{node[vertex](4) {4}
				child[sibling distance =19mm]{node[vertex](8) {8}}
				child[sibling distance=19mm]{node[vertex](9) {9}
					child[sibling distance =19mm]{node[vertex](16) {16}}
					child[sibling distance=19mm]{node[vertex](17) {17}
						child[sibling distance =16mm]{node[vertex](26) {26}}
						child[sibling distance=19mm]{node[vertex](27) {27}}}}}
			child[sibling distance =30mm]{node[vertex](5) {5} 
				child[sibling distance =13mm]{node[vertex](10) {10}}
				child[sibling distance =30mm]{node[vertex](11) {11} 
					child[sibling distance =0.001mm]{node[vertex](18) {18}}
					child[sibling distance =30mm]{node[vertex](19) {19}}}}}
		child{node[vertex](3) {3}
			child[sibling distance =25mm]{node[vertex](14) {14}
				child[sibling distance=0.001mm]{node[vertex](22) {22}}
				child[sibling distance=31mm]{node[vertex](23) {23}}}
			child{node[vertex](15) {15}}};
		\begin{scope}[nodes = {below = 9pt,shift={(0.2cm,0cm)}}]
			\node at (8) {\tiny \makecell[l] {$z_{\text{LP}}=59$\\[2pt] $\CN^{8}_0 = \{3,4,5\}$\\[2pt] $\CN^{8}_1 = \varnothing$} };
			\node at (15) {\tiny \makecell[l] {$z_{\text{LP}} = 62$\\[2pt] $\CN^{15}_0 = \varnothing$\\[2pt] $\CN^{15}_1 = \{4,5,6,7\}$} };
			\node at (16) {\tiny \makecell[l] {$z_{\text{LP}}=62$\\[2pt] $\CN^{16}_0 = \{4,5,6\}$\\[2pt] $\CN^{16}_1 = \{3\}$} };
			\node at (18) {\tiny \makecell[l] {$z_{\text{LP}} = 62$\\[2pt]	$\CN^{18}_0 = \{2,4\}$\\[2pt] $\CN^{18}_1 = \{5,7\}$}	};
			\node at (19) {\tiny \makecell[l] {$z_{\text{LP}} = 65$\\[2pt]	$\CN^{19}_0 = \{4\}$\\[2pt]	$\CN^{19}_1 = \{1,2,5,7\}$}	};
			\node at (22) {\tiny \makecell[l] {$z_{\text{LP}} = 73$\\[2pt] $\CN^{22}_0 = \{2,6\}$\\[2pt] $\CN^{22}_1 = \{4,5\}$} };
			\node at (23) {\tiny \makecell[l] {$z_{\text{LP}} = 62.6$\\[2pt] $\CN^{23}_0 = \{6\}$\\[2pt] $\CN^{23}_1 = \{1,2,4,5\}$} };
			\node at (26) {\tiny \makecell[l] {$z_{\text{LP}} =68$\\[2pt] $\CN^{26}_0 = \{2,4,5\}$\\[2pt] $\CN^{26}_1 = \{3,6,7\}$} };
			\node at (27) {\tiny \makecell[l] {$\text{Infeasible}$\\[2pt] $\CN^{27}_0 = \{4,5\}$\\[2pt] $\CN^{27}_1 = \{1,2,3,6,7\}$} };
		\end{scope}
		\begin{scope}[nodes = {right=11pt, shift={(-0.1cm,0.3cm)}}]
			\node at (1) {\tiny \makecell[l] {${z_{\text{LP}}} = 30.3$\\[2pt] $\CN^1_0 = \CN^1_1 = \varnothing$}};
			\node at (3) {\tiny \makecell[l] {$z_{\text{LP}} = 44.2$\\[2pt]	$\CN^3_0 = \varnothing$\\[2pt] $\CN^3_1 = \{4,5\}$}	};
			\node at (5) {\tiny \makecell[l] {$z_{\text{LP}} = 50$\\[2pt] $\CN^5_0 = \{4\}$\\[2pt]	$\CN^5_1 = \{5\}$}	};
			\node at (11) {\tiny \makecell[l] {$z_{\text{LP}} = 50$\\[2pt] $\CN^{11}_0 = \{4\}$\\[2pt]	$\CN^{11}_1 = \{5,7\}$}	};
			\node at (17) {\tiny \makecell[l] {$z_{\text{LP}} = 56$\\[2pt] $\CN^{17}_0 = \{4,5\}$\\[2pt]	$\CN^{17}_1 = \{3,6,7\}$}	};
		\end{scope}
		\begin{scope}[nodes = {right=11pt, shift={(-0.1cm,0.0cm)}}]
			\node at (10) {\tiny \makecell[l] {$z_{\text{LP}} = 92$\\[2pt] $\CN^{10}_0 = \{4,7\}$\\[2pt]  $\CN^{10}_1 = \{5\}$} };
		\end{scope}
		\begin{scope}[nodes = {right=11pt, shift={(-0.5cm,0.8cm)}}]
			\node at (9) {\tiny \makecell[l] {$z_{\text{LP}} = 56$\\[2pt] $\CN^9_0 = \{4,5\}$\\[2pt]  $\CN^9_1 = \{3\}$} };
		\end{scope}
		\begin{scope}[nodes = {left=11pt, shift={(0.1cm,0.3cm)}}]
			\node at (2) {\tiny \makecell[l] {$z_{\text{LP}} = 50$\\[2pt] $\CN^2_0 = \{4\}$\\[2pt] $\CN^{2}_{1} = \varnothing$}	};
			\node at (4) {\tiny \makecell[l] {$z_{\text{LP}} = 56$\\[2pt] $\CN^4_0 = \{4,5\}$\\[2pt] $\CN^4_1 = \varnothing$} };
			\node at (14) {\tiny \makecell[l] {$z_{\text{LP}} = 56.3$\\[2pt] $\CN^{14}_0 = \{6\}$\\[2pt]	$\CN^{14}_1 = \{4,5\}$} };
		\end{scope}
		\draw (1) -- (2) node [midway,above,sloped,text=blue] {$z_4=0$};
		\draw (1) -- (3) node [midway, above,sloped,text=red] {$z_4=1$};
		\draw (2) -- (4) node [midway,above,sloped,text=blue] {$z_5=0$};
		\draw (2) -- (5) node [midway,above,sloped,text=red] {$z_5=1$};
		\draw (3) -- (14) node [midway,above,sloped,text=blue] {$z_6=0$};
		\draw (3) -- (15) node [midway,above,sloped,text=red] {$z_6=1$};
		\draw (4) -- (8) node [midway,above,sloped,text=blue] {$z_3=0$};
		\draw (4) -- (9) node [midway,above,sloped,text=red] {$z_3=1$};
		\draw (5) -- (10) node [midway,above,sloped,text=blue] {$z_7=0$};
		\draw (5) -- (11) node [midway,above,sloped,text=red] {$z_7=1$};
		\draw (14) -- (22) node [midway,above,sloped,text=blue] {$z_2=0$};
		\draw (14) -- (23) node [midway,above,sloped,text=red] {$z_2=1$};
		\draw (9) -- (16) node [midway,above,sloped,text=blue] {$z_6=0$};
		\draw (9) -- (17) node [midway,above,sloped,text=red] {$z_6=1$};
		\draw (11) -- (18) node [midway,above,sloped,text=blue] {$z_2=0$};
		\draw (11) -- (19) node [midway,above,sloped,text=red] {$z_2=1$};
		\draw (17) -- (26) node [midway,above,sloped,text=blue] {$z_2=0$};
		\draw (17) -- (27) node [midway,above,sloped,text=red] {$z_2=1$};
	\end{tikzpicture}
	\caption{The \BnC search tree of the problem in \cref{example1} with the proposed dominance-based branching applied.}
	\label{figure-newbranch}
\end{figure}
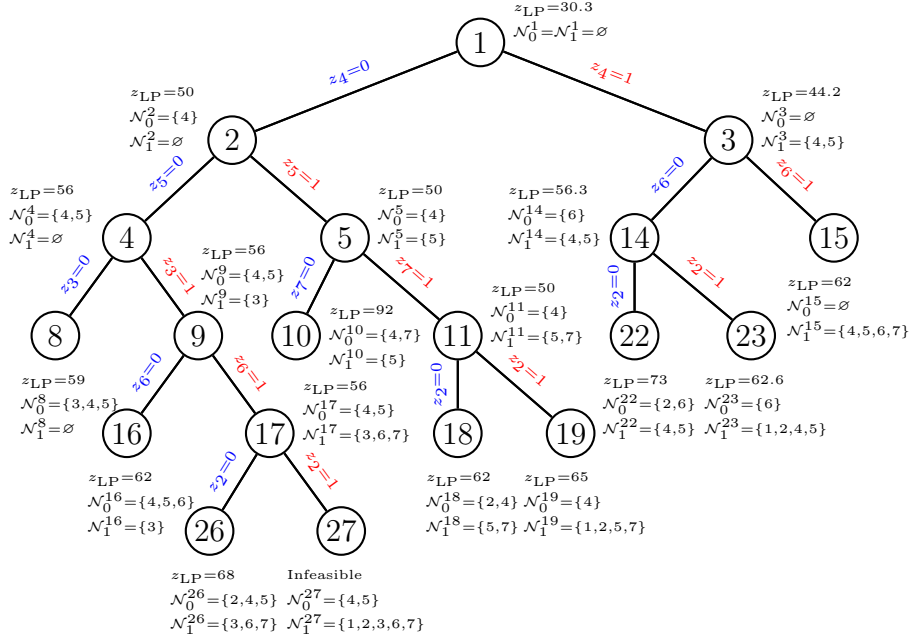

\section{Overlap-oriented node pruning and variable fixing}
\label{sect:consistency branch}

The dominance-based branching creates two subproblems \eqref{strongleftbranch} and \eqref{strongrightbranch} (with more fixed variables)  using the dominance relations in \eqref{newprecedenceconstraints}. 
Such relations are derived from constraints \eqref{milp-consvz}--\eqref{milp-consbound} and the if-then constraints \eqref{cons-ifthen} at the root node of the search tree (that removes the overlaps).
It is possible, however, that at other nodes in the search tree, with the additional fixings of variables to $0$ and $1$, more reductions, i.e., node pruning and variable fixing, can be derived by exploiting the overlap information.
To further enhance the dominance-based branching, in this section, we shall perform {overlap-oriented node pruning and variable fixing} at each 
 node $(\CN_0, \CN_1)$ by considering the set
\begin{equation}\label{FCdef}
	\FC(\CN_0,\CN_1):=\left\{(v,z) \, : \, \eqref{milp-consvz}-\eqref{milp-consbound},~\eqref{cons-ifthen},~ z_i = 0,~\forall~i \in \CN_0,~z_i = 1,~\forall~i \in \CN_1\right\}.
\end{equation}
Here $\CN_0$ and $\CN_1$ are the index sets of variables $z$ fixed to $0$ and $1$, respectively, at the current node.
This set is a variant of the joint mixing set with a knapsack constraint \cite{Kucukyavuz2012,Zhao2017}  that additionally includes the if-then constraints \eqref{cons-ifthen} and the variable fixings at the current node.
Specifically, we can 
\begin{itemize}
	\item [(R1)] prune node $(\CN_0, \CN_1)$ when $\FC(\CN_0,\CN_1)= \varnothing$ is detected; or
	\item [(R2)] fix variables in $\CR_0$ and $\CR_1$ to $0$ and $1$ at node $(\CN_0, \CN_1)$ when $\FC(\CN_0,\CN_1) \neq \varnothing$, 
	where  $\CR_0$ and $\CR_1$ are disjoint subsets of $\CN_f:=[n]\backslash (\CN_0 \cup \CN_1)$ satisfying
	\begin{equation}\label{cond}
		\FC(\CN_0,\CN_1)=\FC(\CN_0\cup \CR_0,\CN_1\cup \CR_1).
	\end{equation}
\end{itemize}
Note that the dominance relations in \eqref{newprecedenceconstraints} can be derived from $\FC(\varnothing, \varnothing)$, and thus the reductions by (R1) and (R2) include the reductions by the dominance relations in \eqref{newprecedenceconstraints} (in the dominance-based branching).
In the following, we shall present exact and approximation approaches to detect when $\FC(\CN_0,\CN_1) = \varnothing$ holds, 
or find the largest subsets $\CR_0$ and $\CR_1$ of $\CN_f$ satisfying \eqref{cond}.

\subsection{The exact approach}
\label{subsect:consistency analysis}

We first consider the overlap-oriented  node pruning in (R1), \ie identify the condition under which $\FC(\CN_0, \CN_1) = \varnothing$ holds.
Let $(v,z) \in \FC(\CN_0, \CN_1)$. 
For $i \in \CN_0$, from \eqref{milp-consvz} and $z_i=0$, we obtain $v \geq \xi^i$, and thus $v \geq \xi^{\CN_0}:= \max_{i \in \CN_0}\{\xi^i \}$ (where the max is taken component-wise and $\xi^{\CN_0} = \boldsymbol{0}$ if $\CN_0 =\varnothing$); and 
for $j \in \CN_1$, from the if-then constraint $z_j =1\Rightarrow v \ngeq \xi^j$ in \eqref{cons-ifthen} and $z_j =1$, we obtain $v \ngeq \xi^j$, or equivalently, $\bigvee_{k \in [m]} \left(v_k < \xi^j_k \right) =1$.
Therefore, 
\begin{align}
	1= \bigwedge_{j \in \CN_1} \bigvee_{k \in [m]} \left(v_k < \xi^j_k \right)  =\bigwedge_{j \in \CN_1} \bigvee_{k \in [m],\,\xi_k^j > \xi_k^{\CN_0}} \left(v_k < \xi^j_k \right)
	, 	\label{cons-dnf1}
\end{align}
where the second equality follows from $v \geq \xi^{\CN_0}$.
Letting $\M_j = \left\{ k \in [m] \, : \, \xi^j_k > \xi_k^{\CN_0} \right\}$ and $\CL = \prod_{j \in \CN_1} \M_j $, then \eqref{cons-dnf1} can be rewritten as
\begin{equation}	\label{cons-dnf}
	\bigvee_{\ell\in\CL} d_{\ell}(v)=1,
\end{equation}
where
\begin{align}
	& d_\ell(v) = \bigwedge_{j \in \CN_1} 
	\left(v_{\ell_j} < \xi^j_{\ell_j} \right).	\label{dlv}
\end{align}
Notice that $\M_j$, $\CL$, and $d_\ell(v)$ indeed depend on $\CN_0$ or $\CN_1$ (or both of them) but we omit this dependence for notation convenience.
The following theorem provides a necessary and sufficient condition for $\FC(\CN_0, \CN_1) \neq \varnothing$.

\begin{Theorem}
	\label{ConditionOneandTwo}
	$\FC(\CN_0, \CN_1) \neq \varnothing$ holds if and only if there exists some $\ell \in \CL$ such that 
	\begin{equation}
		\label{tempkap}
		\sum_{i \in \CN_\ell\cup \CN_1} p_i \leq \epsilon,~\text{where}~
		\CN_\ell = \left\{i\in\CN_f \stt d_{\ell}(\xi^i) = 0 \right\}.
	\end{equation}
\end{Theorem}
\begin{proof}
	\emph{Necessity}. 
	Suppose that $(v, z) \in \FC(\CN_0, \CN_1) \neq \varnothing$.
	By \eqref{cons-dnf}, there must exist some $\ell \in \CL$ such that $d_{\ell}(v) = 1$. 
	For $i \in \CN_\ell$, we have $d_{\ell}(\xi^i)=0$, which together with $d_{\ell}(v) = 1$ and the fact that $d_\ell(v^1) \leq d_\ell(v^2)$ holds for any $v^1, v^2 \in \R^m_+$ with $v^1 \geq v^2$, implies $v \ngeq \xi^i$, and thus $z_i=1$.
	Combining with $z_j =1 $ for $j \in \CN_1$ and \eqref{milp-consknap}, this indicates 
	$$\sum_{i \in \CN_\ell\cup \CN_1} p_i =\sum_{i \in \CN_\ell\cup \CN_1} p_iz_i  \leq \sum_{i=1}^n p_i z_i \leq \epsilon.$$

	\emph{Sufficiency}. Suppose that \eqref{tempkap} holds for some $\ell \in \CL$.
	We define a point $(\hat{v}, \hat{z}) \in \R^m_+ \times \{0,1\}^n $ as follows:
	\begin{align}
		\label{defvhat}
		& \hat{v}_k =\left\{\begin{array}{ll}
			\min\limits_{\ell_j = k,~j \in \CN_1}\{\xi_{\ell_j}^j\}-\delta, &\ \text{if}~k \in \M';\\[10pt]
			\max\limits_{i\in[n]}\left\{\xi^i_k\right\},&\ \text{otherwise},
		\end{array}\right.\forall~k \in [m],  \\
		\label{defzhat}
		& \hat{z}_i = \left\{\begin{array}{ll}
			1, & \ \text{if}~i \in \CN_1\cup \CN_\ell;\\[5pt]
			0, & \ \text{otherwise},
		\end{array}\right.~ \forall~i\in [n],
	\end{align}
	where $\delta>0$ is a sufficiently small value and $\M' = \{ \ell_j \, : \, j \in \CN_1 \}$.
	From $\ell  \in \CL$, we have $\ell_j \in \M_j$ and $\xi^j_{\ell_j} > \xi^{\CN_0}_{\ell_j} \geq 0$ for all $j \in \CN_1$. By \eqref{defvhat} and the fact that $\delta>0$ is sufficiently small, we can derive  $\hat{v} \geq \xi^{\CN_0} \geq 0 $.
	Together with \eqref{tempkap} and \eqref{defzhat}, this implies that \eqref{milp-consknap} and \eqref{milp-consbound} hold at point $(\hat{v}, \hat{z})$.
	In the following, we shall prove $(\hat{v}, \hat{z}) \in \FC(\CN_0, \CN_1)$ by showing that constraints \eqref{milp-consvz} and \eqref{cons-ifthen} hold at $(\hat{v}, \hat{z}) $.

	To prove that \eqref{milp-consvz} holds at $(\hat{v}, \hat{z}) $, it suffices to show $\hat{v} \geq \xi^i $ for every $i \in \CN_0 \cup (\CN_f\backslash\CN_\ell)$.
	For $i \in \CN_0$, $\hat{v} \geq \xi^i$ follows from $\hat{v} \geq \xi^{\CN_0}$.
	Now consider the case $i \in \CN_f\backslash\CN_\ell$.
	First, from the definition of $\hat{v}$ in \eqref{defvhat},  $\hat{v}_k \geq \xi^i_k$ holds for all $k \in [m]\backslash \M'$.
	Second, by the definition of $\CN_\ell$ in \eqref{tempkap} and $i \in \CN_f \backslash \CN_\ell$, we obtain  $d_{\ell}(\xi^i) =\bigwedge_{j \in \CN_1} 
	\left(\xi^i_{\ell_j} < \xi^j_{\ell_j} \right)=1 $, 
	which, together with the fact that $\delta>0$ is a sufficiently small value, 
	implies $\xi^i_{\ell_j} \leq \xi^j_{\ell_j}-\delta$  for all $j \in \CN_1$.
	Therefore, for $k \in \M'$, it also follows $\xi_k^i \leq \min\limits_{\ell_j = k,~j \in \CN_1}\{\xi_{\ell_j}^j\}-\delta = \hat{v}_k$.

	Finally, we show that \eqref{cons-ifthen} holds at point $(\hat{v}, \hat{z})$, which can be done 	by proving $\hat{v} \ngeq \xi^i$ for all $i \in  \CN_1 \cup \CN_\ell$.
	From the definition of $\hat{v}$ in \eqref{defvhat} and $\delta>0$, it is simple to see $d_{\ell}(\hat{v})=1$ and $\hat{v} \ngeq \xi^i$ for all $i \in \CN_1$.
	For $i \in \CN_\ell$, we have $d_{\ell}(\xi^i)  = 0$, which together with  $d_{\ell}(\hat{v})=1$ and
	the fact that $d_\ell(v^1) \leq d_\ell(v^2)$ holds for any $v^1, v^2 \in \R^m_+$ with $v^1 \geq v^2$, implies $\hat{v} \ngeq \xi^i $.
\end{proof}

\cref{ConditionOneandTwo} enables to determine whether $\FC(\CN_0, \CN_1) \neq \varnothing$ by solving an \MILP problem. 
Specifically,  
for $j \in \CN_1$ and $i \in \CN_f$, let $\M_{ji}=\{ k \in \M_j\, : \, \xi_k^j \leq \xi_k^i \}$;
for $j \in \CN_1$ and $k \in \M_j$, let $w_{jk} \in \{0,1\}$ denote whether $v_k < \xi^j_k$ is included in \eqref{dlv}; 
and for $i \in \CN_f$, let $z_i \in \{0,1\}$ denote whether $i \in \CN_\ell$ holds. 
Then it follows  from \cref{ConditionOneandTwo} that
\begin{Corollary}
	\label{coroC}
	$(v,z) \in \FC(\CN_0, \CN_1)$ holds if and only if there exists a vector $w$ for which $(w,z)$ satisfies
	\begin{align}
		& \sum_{k \in \M_j} w_{jk} =1, ~\forall~j \in \CN_1,\label{eq:1}\\
		& \sum_{k \in \M_{ji}}w_{jk} \leq z_{i},~\forall~j\in \CN_1, ~i \in \CN_f,~\label{eq:2}\\
		& \sum_{i \in \CN_f} p_i z_i \leq \epsilon - \sum_{j \in \CN_1}p_j,\label{eq:3}\\
		& w_{jk} \in \{0,1\},~\forall~j \in \CN_1, ~k \in \M_j,~z_{i} \in \{0,1\},~\forall~i \in \CN_f.\label{eq:4}
	\end{align}
\end{Corollary}
\noindent Therefore, to determine whether $\FC(\CN_0, \CN_1) \neq  \varnothing$, we can solve the following \MILP problem 
\begin{equation}\label{MIP}
	o = \min_{w,\,z} \left\{ \sum_{i \in \CN_f } p_i z_i \, : \, \eqref{eq:1},~\eqref{eq:2},~\eqref{eq:4}\right\}.
\end{equation} 
If $o \leq \epsilon - \sum_{j \in \CN_1}p_j$, then $\FC(\CN_0, \CN_1)\neq \varnothing$ and the optimal solution $z$ can define a feasible solution $(v,z)$ of $\FC(\CN_0, \CN_1)$; otherwise, $\FC(\CN_0, \CN_1) = \varnothing$.

Next, we attempt to derive variable fixings for the case $\FC(\CN_0, \CN_1)\neq \varnothing$. 
Let $(v,z) \in \FC(\CN_0, \CN_1)$.
Observe that $z_i=0$ (respectively, $z_i=1$) holds for all $(v,z) \in \FC(\CN_0, \CN_1)$ if and only if $\FC(\CN_0, \CN_1\cup\{i\}) = \varnothing$ (respectively, $\FC(\CN_0\cup\{i\}, \CN_1) = \varnothing$) holds for $i \in \CN_f$.
Hence, the largest subsets $\CR_0$ and $\CR_1$ satisfying \eqref{cond} can be written as 
\begin{equation*}
	\CR_0 = \{ i \in \CN_f\, : \, \FC(\CN_0, \CN_1\cup\{i\}) = \varnothing  \} ~\text{and}~\CR_1 = \{ i \in \CN_f\, : \, \FC(\CN_0\cup\{i\}, \CN_1) = \varnothing  \}.
\end{equation*}
As a result, determining the largest subsets $\CR_0$ and $\CR_1$ satisfying \eqref{cond} 
can be done by solving $2|\CN_f|$ \MILPs of the form \eqref{MIP}.
Notice that for $i \in \CN_f\backslash (\CR_0 \cup \CR_1)$, both $\FC(\CN_0\cup \CR_0 \cup \{i\},\CN_1\cup \CR_1)\neq \varnothing$ and $\FC(\CN_0\cup \CR_0,\CN_1\cup \CR_1\cup\{i\})\neq \varnothing$ must hold. 
Therefore,
\begin{Remark}\label{re:nofixing}
	If the overlap-oriented variable fixing is performed to find the largest subsets $\CR_0$ and $\CR_1$ satisfying \eqref{cond} at all nodes of the search tree, then no overlap-oriented node pruning can be performed. 
	That is, for any node $(\CN_0, \CN_1)$ in the search tree, it must follow $\FC(\CN_0, \CN_1) \neq \varnothing$.
\end{Remark}
\begin{Example}[continued]
	Applying  the dominance-based branching \rev{with the} overlap-oriented variable fixing to the problem in \cref{example1}, we obtain the search tree in \cref{figure-pra}.
	All variables fixed to $0$ or $1$ by overlap-oriented variable fixing are underlined in \cref{figure-pra}.
	
	Let us perform variable fixing at node $5$ where, at the beginning, $\CN_0 := \CB_0 = \{4\}$ and $\CN_1 := \CB_1 = \{5\}$.
	By simple computations, we have $\M_5 = \{1\}$, $\M_{51}=\M_{52}=\M_{53}=\varnothing$, and $\M_{56}=\M_{57}=\{1\}$.
	Hence, problem \eqref{MIP} reduces to 
	\begin{equation*}\small
		\min\left\{ \frac{1}{7}\!\!\sum_{i\in [7]\backslash \{4,5\}} \!\!\!\!\!\! z_i : w_{51} = 1, w_{51} \leq z_6, w_{51} \leq z_7, w_{51} \in \{0,1\}, z_i \in \{0,1\}, \forall\,i \in [7]\backslash\{4,5\}\right\}.
	\end{equation*}
	As the optimal value of the above problem is $\frac{2}{7} \leq \epsilon - \sum_{j \in \CN_1} p_j = \frac{3}{7}$, from \cref{coroC} $\FC(\{4\}, \{5\}) \neq \varnothing$ holds.
	Similarly, we can show that $\FC(\{4,6\},\{5\})=\FC(\{4,7\},\{5\}) = \FC(\{4\},\{2,5\}) =  \FC(\{4\},\{3,5\}) =\varnothing$, $\FC(\{1, 4\},\{5\}) \neq \varnothing$, and $\FC(\{4\},\{1,5\}) \neq \varnothing$. Thus, we can fix $z_2=z_3=0$ and $z_6 = z_7 = 1$ at node $5$.
	\begin{figure}[!h]
		\everymath{\scriptscriptstyle}
		\centering
			\begin{tikzpicture}
			[level distance=12mm,thick,
			level 1/.style={sibling distance=50mm},
			level 2/.style={sibling distance=25mm},
			level 3/.style={sibling distance=20mm}]
			\node[vertex](1) {1}
			child{node[vertex](2) {2}
				child{node[vertex](4) {4}
					child{node[vertex](8) {8}}
					child{node[vertex](9) {9}}}
				child{node[vertex](5) {5}}}
			child{node[vertex](3) {3}
				child{node[vertex](14) {14}}
				child{node[vertex](15) {15}}};
			\begin{scope}[nodes = {below = 9pt,shift={(0.1cm,0cm)}}]
				\node at (5) {\tiny \makecell[l] {$z_{\text{LP}} = 62$\\[2pt] $\CN^5_0 = \{\underline{2},\underline{3},4\}$\\[2pt]	$\CN^5_1 = \{5,\underline{6},\underline{7}\}$}	};
				\node at (8) {\tiny \makecell[l] {$z_{\text{LP}}=59$\\[2pt] $\CN^{8}_0 = \{3,4,5\}$\\[2pt] $\CN^{8}_1 = \varnothing$} };
				\node at (9) {\tiny \makecell[l] {$z_{\text{LP}} = 62$\\[2pt] $\CN^9_0 = \{4,5,\underline{6}\}$\\[2pt]  $\CN^9_1 = \{\underline{1},\underline{2},3\}$} };
				\node at (14) {\tiny \makecell[l] {$z_{\text{LP}} = 73$\\[2pt]	$\CN^{14}_0 = \{2, 6\}$\\[2pt]	$\CN^{14}_1 = \{\underline{3},4,5\}$} };
				\node at (15) {\tiny \makecell[l] {$z_{\text{LP}} = 62$\\[2pt] $\CN^{15}_0 = \{\underline{1},2,\underline{3}\}$\\[2pt] $\CN^{15}_1 = \{4,5,6,7\}$} };
			\end{scope}
			\begin{scope}[nodes = {right=11pt, shift={(0.0cm,0.3cm)}}]
				\node at (1) {\tiny \makecell[l] {${z_{\text{LP}}} = 30.3$\\[2pt] $\CN^1_0 = \CN^1_1 = \varnothing$}};
				\node at (3) {\tiny \makecell[l] {$z_{\text{LP}} = 49$\\[2pt]	$\CN^3_0 = \{\underline{2}\}$\\[2pt] $\CN^3_1 = \{4,5\}$}	};
			\end{scope}
			\begin{scope}[nodes = {left=11pt, shift={(0.1cm,0.3cm)}}]
				\node at (4) {\tiny \makecell[l] {$z_{\text{LP}} = 56$\\[2pt] $\CN^4_0 = \{4,5\}$\\[2pt] $\CN^4_1 = \varnothing$} };		
				\node at (2) {\tiny \makecell[l] {$z_{\text{LP}} = 50$\\[2pt] $\CN^2_0 = \{4\}$\\[2pt] $\CN^{2}_{1} = \varnothing$}	};
			\end{scope}
			\draw (1) -- (2) node [midway,above,sloped,text=blue] {$z_4=0$};
			\draw (1) -- (3) node [midway, above,sloped,text=red] {$z_4=1$};
			\draw (2) -- (4) node [midway,above,sloped,text=blue] {$z_5=0$};
			\draw (2) -- (5) node [midway,above,sloped,text=red] {$z_5=1$};
			\draw (3) -- (14) node [midway,above,sloped,text=blue] {$z_6=0$};
			\draw (3) -- (15) node [midway,above,sloped,text=red] {$z_6=1$};
			\draw (4) -- (8) node [midway,above,sloped,text=blue] {$z_3=0$};
			\draw (4) -- (9) node [midway,above,sloped,text=red] {$z_3=1$};
		\end{tikzpicture}
		\caption{The \BnC search tree of the problem in \cref{example1} with the proposed dominance-based branching \rev{with the} overlap-oriented variable fixing applied.}
		\label{figure-pra}
	\end{figure}
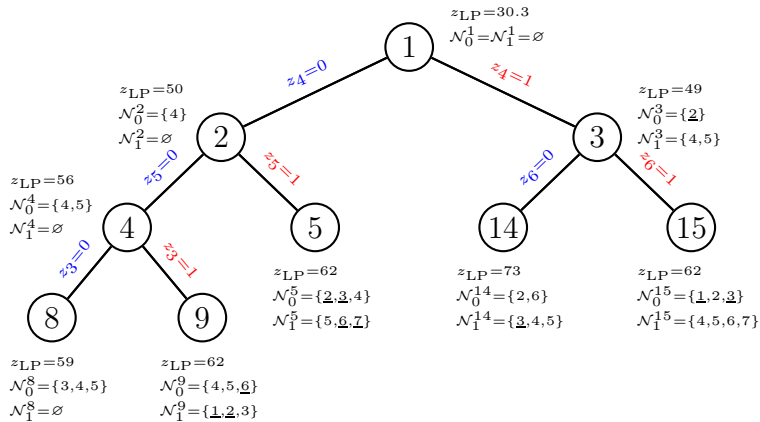
\end{Example}

The above example \rev{helps demonstrate} the advantage of the overlap-oriented variable fixing.
For instance, for node $5$ in the search tree in \cref{figure-pra}, the overlap-oriented variable fixing improves the \LP relaxation bound from $50$ to $62$ and avoids further branching.
Overall, applying the dominance-based branching \rev{with the} overlap-oriented variable fixing, only $9$ nodes are explored while applying the vanilla dominance-based branching, $19$ nodes are explored; see \cref{figure-newbranch,figure-pra}. 

To apply the overlap-oriented node pruning and variable fixing, we need to determine whether the system of the form \eqref{eq:1}--\eqref{eq:4} has a feasible solution (which can be done by solving \MILP problems of the form \eqref{MIP}).
Unfortunately, the following theorem shows that there does not exist a polynomial-time exact algorithm for solving the above problem.

\begin{Theorem}
	\label{nphard}
	Given $\CN_0, \CN_1 \subseteq[n]$, determining whether \eqref{eq:1}--\eqref{eq:4} has a feasible solution is strongly NP-complete.
\end{Theorem}
\begin{proof}
	We shall prove the strong NP-completeness of the problem of deciding whether \eqref{eq:1}--\eqref{eq:4} has a feasible solution
	by establishing a polynomial-time reduction from the strongly NP-hard problem: set covering (SC) problem \cite{Garey1978}.
	We first introduce the SC problem: 
	given $t$ subsets $\J_1, \J_2, \ldots, \J_t$ of $\J$ where $\J$ is a finite set of $r$ elements, 
	does there exist $\CS \subseteq [t]$ such that $|\CS| \leq B $ and $\cup_{i\in \CS}\J_i = \J$? 
	The SC problem is equivalent to deciding whether the following system
	\begin{equation}
		\label{Wdef}
		\sum_{i =1}^t z_i \leq B,~\sum_{i \in \I_j} z_i  \geq 1,  ~\forall~j \in \J,~z_i \in \{0,1\}, ~\forall~i \in [t],
	\end{equation}
	has a feasible solution $z$, where $\I_j = \{ i \in [t]\, : \, j  \in \J_i \}$.
	For notations purpose, we denote $\J=\{t+1, t+2, \ldots, t+r\}$.
	
	Given any instance of the SC problem, we construct an instance of \rev{the problem of determining} whether \eqref{eq:1}--\eqref{eq:4} has a feasible solution as follows:
	\begin{itemize}
		\item [(i)]  $n :=  t + r$, $m := t$, $p_i  := \frac{1}{t+r}$, $i \in [t+r]$,  $\epsilon := \frac{B+r}{t+r}$, $\CN_0:=\varnothing$, and $\CN_1:=\J =\{t+1, t+2, \ldots, t+r\}$;
		\item [(ii)] For each $i \in [t]$, we define $\xi^{i} := \boldsymbol{e}_{i} $ where $\boldsymbol{e}_{i}$ is the $i$-th unit vector of dimension $t$, and
		for each $j \in \{t+1, t+2, \ldots,t+r\}$, we define $\xi^j$ as follows:
		\begin{equation}
			\label{xiidef}
			\xi^{j}_k :=\left\{\begin{array}{ll}
				1,  & \ \text{if} ~k \in \I_{j} ; \\
				0, &\ \text{otherwise},
			\end{array}\right.~\forall~k \in [t].
		\end{equation}
	\end{itemize}
	By the definitions of $p_i$, $\epsilon$, $\CN_0$, and $\CN_1$, \eqref{eq:3} reduces to
	\begin{equation}
		\label{tmeq:3}
		\sum_{i =1}^t z_i \leq B.
	\end{equation}
	As $\CN_0 = \varnothing$, it follows $\xi^{\CN_0}=\boldsymbol{0}$.
	For $j \in \CN_1$, it follows from \eqref{xiidef} that $\M_{j} = \{ k \in [t] \, : \, \xi^{j}_k > 0 \} =\{ k \in [t] \, : \, \xi^{j}_k = 1 \}   = \I_j$.
	Thus \eqref{eq:1} and \eqref{eq:4} reduce to 
	\begin{align}
			& \sum_{i \in \I_j} w_{ji} =1,~\forall~j \in \CN_1,	\label{tmeq:1}\\
			& w_{ji} \in \{0,1\},~\forall~j \in \CN_1, ~i \in \I_j,~z_{i} \in \{0,1\},~\forall~i \in [t].	\label{tmeq:4}
	\end{align}
	For $j \in \CN_1$ and $i \in [t]$, we have 
	$\M_{ji} = \{ k \in \M_j \, : \, \xi_k^{j}\leq \xi_k^i \} = \{ k \in \I_j \, : \, \xi_k^{j}\leq \xi_k^i \} =  \{ k\in \I_j\, : \, 1\leq \xi_k^i \} $, which, together with $\xi^i = \boldsymbol{e}_i$, implies 
	\begin{equation*}
		\label{mjidef}
		\M_{ji} =\left\{\begin{array}{ll}
			\{i\},  & \ \text{if} ~i \in \I_j ; \\
			\varnothing, &\ \text{otherwise}.
		\end{array}\right.
	\end{equation*}
	Therefore, \eqref{eq:2} reduces to 
	\begin{equation}
			\label{tmeq:2}
		w_{ji} \leq z_i,~\forall~j \in \CN_1,~i \in \I_j.
	\end{equation}
	Observe that $(\rev{w},z)$ satisfies \eqref{tmeq:1}--\eqref{tmeq:2} if and only if $\sum_{i \in \I_{j}} z_i  \geq 1$ and $z_i \in \{0,1\}$ hold for all  $j \in \CN_1$ and  $i \in [t]$, respectively.
	Therefore, \eqref{tmeq:3}--\eqref{tmeq:2} has a feasible solution if and only if \eqref{Wdef} has a feasible solution, which completes the proof.
\end{proof}

Due to the negative result in \cref{nphard}, we shall develop an efficient heuristic algorithm to apply the overlap-oriented node pruning and variable fixing in the following.

\subsection{The polynomial-time approximation approach}
\label{subsect:polynomial branch}
In this subsection, we present a polynomial-time approximation approach to apply overlap-oriented node pruning and variable fixing.
The proposed approach iteratively performs the following two steps until no more reduction is found: (i) deriving lower bounds for variables $v$ from the constraints in $\FC(\CN_0, \CN_1)$, and (ii) using the derived lower bounds to identify when $\FC(\CN_0, \CN_1)=\varnothing$ or determine $\CR_0$ and $\CR_1$ satisfying \eqref{cond}.
Before going into the details, we note that if $\epsilon - \sum_{i \in \CN_1} p_i < 0$, it must follow  $\FC(\CN_0, \CN_1)=\varnothing$, 
and therefore, to perform the reductions, it suffices to consider the case
 \begin{equation}
 		\label{A1}
 		\epsilon - \sum_{i \in \CN_1} p_i \geq 0.
 \end{equation}

\subsubsection{Deriving lower bounds for variables $v$}
\label{subsubsect:lower bound}

As noted in \cref{subsect:consistency analysis}, 
using constraints \eqref{milp-consvz} and $z_i=0 $ for $i \in \CN_0$ in $\FC(\CN_0, \CN_1)$, we obtain the following lower bounds for variables $v$: 
\begin{equation}
	\label{lb0}
	v_k \geq  \xi^{\CN_0}_k,~\forall~k \in [m],
\end{equation}
where we recall that $\xi^{\CN_0}_k = \max_{i \in \CN_0}\{\xi_k^i\}$.
Another way to derive lower bounds for variables $v$ is to use the technique as in \cref{subsect:preprocess}.
Specifically,
for each $k \in [m]$,  let $\pi_k(1), \pi_k(2), \ldots, \pi_k(|\CN_f|) $ be a permutation of $\CN_f$ satisfying
\begin{equation}
	\label{sortings}
	\begin{aligned}
		& \xi^{\pi_k(1)}_{k} \geq  \xi^{\pi_k(2)}_{k} \geq\cdots 
		\geq \xi^{\pi_k(|\CN_f|)}_{k},
	\end{aligned}
\end{equation}
and let
\begin{equation}
	\label{defell}
	\tau_k := \min \left\{s\in\left[|\CN_f|\right] \stt \sum_{i = 1}^s p_{\pi_k(i)} > \epsilon-\sum_{i\in\CN_1}p_i \right\}.
\end{equation}
By equation \eqref{A1}, we must have $\tau_k\geq 1$.
Using the same technique in \cref{subsect:preprocess}, the following must hold:
\begin{equation}
	\label{lb1}
	v_k \geq \xi_k^{\pi_k(\tau_k)},~\forall~k \in [m].
\end{equation}
Combining \eqref{lb0} and \eqref{lb1}, we can obtain tighter lower bounds for variables $v$:
\begin{equation}
	\label{lb}
	v_k \geq \xi_k^{\CN_0, \CN_1}:=  \max \left\{\xi^{\CN_0}_k,\xi_k^{\pi_k(\tau_k)}\right\},~\forall~k \in [m].
\end{equation}
It is worthy noting that the lower bounds $\{\xi_k^{\CN_0, \CN_1}\}$ depend on both $\CN_0$ and $\CN_1$: the larger the $\CN_0$ and $\CN_1$, the tighter the lower bounds $\{\xi_k^{\CN_0, \CN_1}\}$.

Obviously, the computation of $\xi_k^{\CN_0}$, $k \in [m]$, in \eqref{lb0} can be done in the complexity of $\mathcal{O}(m|\CN_0|)$.
As for the computation of $\xi_k^{\pi_k(\tau_k)}$, $k \in [m]$, we need to, for each $k\in [m]$, sort  $\xi_k^i$, $i \in \CN_f$, satisfying \eqref{sortings}  and then determine $\tau_k$ satisfying \eqref{defell}, which can be implemented with the complexity of  $\mathcal{O}(m|\CN_f|\log(|\CN_f|))$.
Therefore, the complexity for the computation of lower bounds for variables $v$ is  $\mathcal{O}(m|\CN_0|+m|\CN_f|\log(|\CN_f|))$.

\subsubsection{Node pruning and variable fixing}
\label{subsubsect:variable fixing}
Using the lower bounds of variables $v$ in \eqref{lb}, we are able to 
give a sufficient condition under which node $(\CN_0, \CN_1)$ can be pruned, as detailed in the following proposition.
\begin{Proposition} 
	\label{Assump}
	Let $\CN_0$, $\CN_1 \subseteq [n]$ be such that $\CN_0 \cap \CN_1 = \varnothing$.
	If $\xi^j \leq \xi^{\CN_0, \CN_1}$ for some $j \in \CN_1$, then
	$\FC(\CN_0, \CN_1)=\varnothing$.
\end{Proposition}
\begin{proof}
	Let $j \in \CN_1$ such that $\xi^j \leq  \xi^{\CN_0, \CN_1}$. 
	Suppose, otherwise, that $(v,z) \in \FC(\CN_0, \CN_1)$. 
	By \eqref{lb}, we have $v \geq \xi^{\CN_0, \CN_1}\geq \xi^j$, which, together with $z_j\in \{0,1\}$ and \eqref{cons-ifthen}, implies that $z_j = 0$. 
	However, this contradicts with $z_j =1$ (as $j \in \CN_1$) and thus  
	$\FC(\CN_0, \CN_1) = \varnothing$.
\end{proof}

Next, we attempt to perform variable fixing for the case that the condition in \cref{Assump} does not hold, i.e., 
\begin{equation}
	\label{assumption2}
	\xi^j \nleq \xi^{\CN_0, \CN_1},~\forall~j \in \CN_1.
\end{equation}
To begin with, we note that using the lower bounds of variables $v$ in \eqref{lb}, constraints $v \ngeq \xi^j $, or equivalently, $\bigvee_{k \in [m]} \left(v_k < \xi^j_k \right) = 1$  in $\FC(\CN_0, \CN_1)$ can be simplified as  
\begin{equation}
	\label{civ1}
	c_j(v): = \bigvee_{k \in\bar{\M}_j} \left(v_k < \xi^j_k \right)=1,~\forall~j \in \CN_1,
\end{equation}
where 
\begin{equation}
	\label{defMi1}
	\bar{\M}_j := \left\{k \in [m] \stt \xi^j_k > \xi_k^{\CN_0, \CN_1}\right\}.
\end{equation}
Let 
\begin{equation}
	\label{defS01}
	\begin{aligned}
	& \CR_1 := \left\{i\in\CN_f \stt c_j(\xi^i) = 0~\text{holds for some}~j\in\CN_1\right\},\\
	& \CR_0 :=\left\{i\in\CN_f \stt p_i > \epsilon-\sum_{j \in\CN_1\cup \CR_1}p_j
	~\text{or}~
	\xi^{i} \leq \xi^{\CN_0, \CN_1}
	\right\}.
	\end{aligned}
\end{equation}
Note that compared with the definition $\M_j$ in \cref{subsect:consistency analysis}, we use the tighter lower bounds $\xi_k^{\CN_0, \CN_1}$ for the definition of $\bar{\M}_j$, and thus $\bar{\M}_j\subseteq \M_j$.
Also note that the tighter the lower bounds $\xi_k^{\CN_0, \CN_1}$, the smaller the $\bar{\M}_j$, and the larger the $\CR_0$ and $\CR_1$.

The following proposition shows that for node $(\CN_0, \CN_1)$, variables $z_i$ for $i \in \CR_0$ can be fixed to zero; and variables $z_i$ for $i \in \CR_1$ can be fixed to one.

\begin{Proposition}
	\label{fixvars}
	Let $\CN_0$, $\CN_1 \subseteq [n]$ be such that $\CN_0 \cap \CN_1 = \varnothing$ and \eqref{assumption2} hold, and $\CR_0$ and $\CR_1$ be defined in \eqref{defS01}.
	Then $\FC(\CN_0, \CN_1) = \FC(\CN_0\cup\CR_0, \CN_1\cup\CR_1)$.
\end{Proposition}

\begin{proof}
	We shall show that for any $(v,z) \in \FC(\CN_0, \CN_1)$, $z_i=1$ and $z_i=0$ hold for $i \in \CR_1$ and $i \in \CR_0$, respectively. 
	\begin{itemize}
	\item[(i)] For $i \in \CR_1$, $c_j(\xi^i) = 0$ holds for some $j\in\CN_1$.
	By \eqref{civ1}, we have $c_j(v) = 1$.
	Then $v \ngeq \xi^i$ follows from the definition of $c_j(\cdot)$ in \eqref{civ1}, that is, for any $v^1, v^2 \in \R^m_+$ with $v^1 \leq v^2$, it follows $c_j(v^2) \leq c_j(v^1)$.
	Together with $z_i \in \{0,1\}$ and \eqref{milp-consvz}, this implies $z_i =1$.
	\item[(ii)] For $i \in \CR_0$,
	if $p_i > \epsilon-\sum_{j \in\CN_1\cup\CR_1}p_j$, then it follows from \eqref{milp-consknap} and $z_j = 1$ for all $j \in \CN_1\cup \CR_1$ that $z_i = 0$.
	Otherwise, $\xi^i \leq \xi^{\CN_0, \CN_1}$, and by  \eqref{lb}, $ v \geq  \xi^{\CN_0, \CN_1} \geq \xi^i$ holds.
	Together with  $z_{i} \in \{0,1\}$ and \eqref{cons-ifthen}, this implies $z_i = 0$. \qedhere
	\end{itemize}
\end{proof}

Notice that using \cref{Assump} to determine whether $\FC(\CN_0, \CN_1) = \varnothing$ and computing $c_j(v)$, $j \in \CN_1$, can all be done in the complexity of  $\CO(m|\CN_1|)$. 
The computations of $\CR_0$ and $\CR_1$ in \eqref{defS01} can be implemented with the complexity of $\CO(m|\CN_1||\CN_f|)$. 
Therefore, the overall complexity of performing node pruning and variable fixing is $\CO(m|\CN_1||\CN_f|)$.

\subsubsection{The overall algorithmic framework}
\label{subsubsect:polynomial reduction}

After performing the variable fixing in \cref{subsubsect:variable fixing}, we may compute tighter lower bounds for variables $v$ using the procedure \cref{subsubsect:lower bound} again, which, in turn, opens up new possibilities to detect more reductions using the procedure in \cref{subsubsect:variable fixing}.
We illustrate this using the following example.
\begin{Example}\label{algex}
	Let us consider node $5$ in \cref{figure-newbranch} where the dominance-based branching is applied to solve the problem in \cref{example1}. 
	At the beginning, $\CN_0 := \CB_0 = \{4\}$ and $\CN_1 := \CB_1 = \{5\}$.
	Applying the procedure in \cref{subsubsect:lower bound}, we obtain the lower bounds $\xi^{\CN_0, \CN_1} = (5,\,2,\,6)^\top$ for variables $v$.
	Applying the procedure in \cref{subsubsect:variable fixing}, we obtain $\CR_0 = \varnothing$ and  $\CR_1 = \{6,7\}$, and thus $\FC(\{4\}, \{5\})=\FC(\{4\}, \{5,6,7\})$.
	
	Now letting $\CN_0 :=\{4\}$ and $\CN_1 := \{5,6,7\}$, applying the procedure in \cref{subsubsect:lower bound} again, we obtain the tighter lower bounds $\xi^{\CN_0, \CN_1} = (5,\,2,\,10)^\top$ for variables $v$.
	Similarly, applying the procedure in \cref{subsubsect:variable fixing} again, we can obtain $\CR_0 = \{2,3\}$ and $\CR_1 = \varnothing$, and thus $\FC(\{4\}, \{5,6,7\})=\FC(\{2,3,4\}, \{5,6,7\})$.
\end{Example}

\cref{algex} offers a hint to the algorithmic design for the  overlap-oriented node pruning and variable fixing.
Specifically, we can iteratively apply the procedures in \cref{subsubsect:lower bound,subsubsect:variable fixing} to detect reductions for node $(\CN_0, \CN_1)$ until no more reduction is found (i.e., either the node is pruned or $\CR_0 =\CR_1 =\varnothing$). 
The details are summarized in \cref{algorithm-p}.
Note that in one iteration of \cref{algorithm-p}, either the procedure is terminated or at least one more variable is fixed. 
Therefore, the number of iterations in \cref{algorithm-p} is at most $|\CN_f|$, and the worst-case complexity of \cref{algorithm-p} is polynomial.
Moreover, as will be demonstrated in  \cref{sect:computational results},  \cref{algorithm-p} is indeed competitive with the exact approach in \cref{subsect:consistency analysis} in terms of reducing the tree size while enjoying a high computational efficiency.

\begin{algorithm}[!htbp]
	\caption{An iterative procedure for overlap-oriented node pruning and variable fixing}
	\label{algorithm-p}
	\KwIn{Node $(\CN_0, \CN_1)$.}
	\Repeat{no more reduction is found}
	{
		If $\epsilon - \sum_{i \in \CN_1} p_i < 0$, \textbf{stop} and claim that the node is infeasible\;
		Compute the lower bounds $\xi^{\CN_0,\CN_1}$ as in \eqref{lb}\;
		If $\xi^j \leq \xi^{\CN_0,\CN_1}$  for some $j \in\CN_1$, \textbf{stop} and claim that the node is infeasible\;
		Compute $\CR_0$ and $\CR_1$ as  in \eqref{defS01}\;
		Update $\CN_0 \leftarrow \CN_0 \cup \CR_0$ and $\CN_1 \leftarrow \CN_1 \cup \CR_1$\;
	}
\end{algorithm}

\section{Computational results}
\label{sect:computational results}
In this section, we present computational results to illustrate the effectiveness of the proposed dominance-based branching \rev{with the} overlap-oriented node pruning and variable fixing techniques.
\rev{To do this, we first conduct computational experiments to compare the proposed dominance-based branching (with the overlap-oriented node pruning and variable fixing techniques) with the state-of-the-art approaches in \citet{Ruszczynski2002} and \citet{Luedtke2010a} in \cref{subsect:stateoftheart}.}
\rev{Then, to gain more insight into the proposed approaches, we compare the (vanilla) dominance-based branching with the direct use of dominance inequalities in \cref{subsect:branch}, and evaluate the performance effect of the overlap-oriented node pruning and variable fixing in \cref{subsect:fixing}.}
\rev{Finally, we compare the performance of the proposed dominance-based branching with the overlap-oriented node pruning and variable fixing techniques for \CCPs with different dimensions $m$ of the random vector $\xi$ in \cref{subsect:effectm}.}

The proposed methods were implemented in C language and linked to the state-of-the-art open source \MILP solver \rev{\scipv\cite{solverscip2024}, using \soplexv} as the \LP solver.
SCIP includes a routine of domain propagation methods \cite{Achterberg2007} that is applied at each node of the search tree to tighten the variable bounds (including variable fixings) or detect infeasible subproblems.
Therefore, we implemented the proposed dominance-based branching \rev{with the} overlap-oriented node pruning and variable fixing techniques as domain propagation methods. 
With this implementation, other sophisticated components of SCIP including  the default fine-tuned branching strategy (called \emph{hybrid branching}) can still be used. 
Moreover, as demonstrated in \cite{Gamrath2014,Gamrath2020a}, when selecting the branching variables, the default branching strategy of SCIP will invoke the domain propagation methods; and thus, this implementation also enables the proposed dominance-based branching \rev{with the} overlap-oriented node pruning and variable fixing techniques to guide the selection of branching variables.

In order to improve the overall solution efficiency, we follow \cite{Luedtke2010a} to apply the preprocessing technique to strengthen formulation \eqref{milp} of problem \eqref{ccps}. 
Specifically, since $\xi^0$ are the lower bounds  for variables $v$ (see \cref{lemma:defxi0}), we can remove constraint $v_k \geq \xi_k^i (1-z_i)$ from the problem formulation if $\xi_k^i \leq \xi_k^0$ and strengthen it as  $v_k \geq \xi^i_k - ({\xi}^i_k-\xi^0_k) z_i$  otherwise.

In our computational study, we consider three \CCPs studied in the literature, which are the  \rpp problem \cite{Gurvich2010,Luedtke2014}, \mpp problem  \cite{Dey2023}, and \lsp problem \cite{Beraldi2002}.
Our testset consists of $135$ \rpp instances, $135$ \mpp instances, and $180$ \lsp instances.
\rev{Note that, in \cref{subsect:effectm}, we use a testset of $270$ \mpp instances to compare the performance of the proposed approaches on instances with different dimensions $m$ of the random vector $\xi$.}
The descriptions of the problems and the instance construction procedures are  provided in Appendix \ref{sect:appendix}.
Except where explicitly stated, all computations were performed on a cluster of Intel(R) Xeon(R) Gold 6140 CPU @ 2.30GHz computers running Linux,
with a time limit of $4$ hours and a relative gap of $0 \%$, \rev{where the relative gap is defined as \citep{MIPLIB2010}:
\begin{equation*}
	\text{RG} = 
	\left\{\begin{array}{ll}
		\frac{\text{UB}-\text{LB}}{\min\{|\text{UB}|, |\text{LB}|\}} \times 100\%,  & \ \text{if}~{\text{LB}} \cdot \text{UB}> 0; \\
		\infty, &\ \text{otherwise},
	\end{array}\right.
\end{equation*}
and $\text{LB}$ and $\text{UB}$ denote the lower and upper bounds returned by SCIP, respectively.
By definition, a relative gap of $0 \%$ implies $\text{LB}=\text{UB}$, that is, an optimal solution has been obtained by SCIP.}
Throughout this section, all averages are reported as the shifted geometric mean with shifts of $1$ second and $100$ nodes for the CPU time and the number of explored nodes, respectively.
The shifted geometric mean of values $x_1,x_2,\ldots,x_n$ with shift $s$ is defined as $\prod^n_{k=1} (x_k +s)^{1/n}-s$; see \cite{Achterberg2007}.

\subsection{\rev{Comparison with the state-of-the-art approaches in \citet{Ruszczynski2002} and \citet{Luedtke2010a}}}\label{subsect:stateoftheart}

\rev{In this subsection, we demonstrate the efficiency of the proposed dominance-based branching (with the overlap-oriented node pruning and variable fixing techniques) by comparing it with the state-of-the-art approaches in \citet{Ruszczynski2002} and \citet{Luedtke2010a}. In particular, we compare the following three settings:
	\begin{itemize}
		\item[$\bullet$] \tblMIX: solving formulation \eqref{milp} using  the \BnC algorithm with the mixing cuts of \citep{Luedtke2010a}\footnote{\rev{\citet{Kucukyavuz2012,Abdi2016,Zhao2017} developed other classes of valid inequalities for problem \eqref{milp}. 
		However, the results in \citep{Luedtke2010a} indicate that the mixing cuts can already provide a substantial portion of the strength of the \LP relaxation of formulation \eqref{milp}, and thus we do not implement other valid inequalities in our experiment.}};
		\item[$\bullet$] \tblRF: \tblMIX with the dominance inequalities \cite{Ruszczynski2002} (when adding the dominance inequalities \eqref{precedenceconstraints} into formulation \eqref{milp}, we only add the nonredundant ones; see \cref{compactineq});
		\item[$\bullet$] \tblPBS: solving formulation \eqref{milp} using the proposed dominance-based branching (with the overlap-oriented node pruning and variable fixing techniques) where the mixing cuts were also implemented.
\end{itemize}}

\cref{solver} summarizes the computational results of the three problems\footnote{Throughout, we report the aggregated results (for each of the three problems). 
Detailed statistics of instance-wise computational results can be found in the online supplement available at \url{https://drive.google.com/file/d/1hZnv0jgoFUjyIS7Fwyo6bA_6p1tu9yil/view?usp=share_link}.}.
The rows ``$\geq s$" collect the subsets of instances that can be solved by at least one setting within the time limit and for which the number of explored nodes returned by at least one of the three settings is at least $s$.
With the increasing $s$, this provides a hierarchy of subsets of increasing difficulty (as the tree size is large).
In column \texttt{\#}, we report the number of instances that can be solved to optimality within the time limit by at least one setting.
For each setting, we report the number of solved instances (\tblS), the average CPU time in seconds (\tblT),
and the average number of explored nodes (\tblN). 
\rev{For setting \tblPBS, we additionally report the average CPU time spent in the implementation of the proposed dominance-based branching with the approximation approach of the overlap-oriented node pruning and variable fixing ($\tblTd$), 
the average number of variables fixed by the dominance-based branching with the overlap-oriented variable fixing per node (\tbldomred), 
and the average number of nodes pruned by the overlap-oriented node pruning using the approximation approach (\tblcutoff).}

\rev{We first compare the performance of settings \tblMIX and \tblRF.
	As shown in \cref{solver}, the number of explored nodes returned by \tblRF is smaller than that of \tblMIX, showing that the dominance inequalities can indeed contribute to eliminate the search tree size.
	For \lsp instances where a large tree reduction by the dominance inequalities can be observed, \tblRF effectively outperforms \tblMIX. 
	For \rpp and \mpp instances, only a slight tree reduction is observed and thus \tblRF cannot outperform \tblMIX; indeed, for \rpp instances, \tblRF is even outperformed by \tblMIX due to a large problem size caused by the addition of the dominance inequalities.}

Next, we compare the performance of settings \tblPBS and \tblRF.
\rev{We can observe from \cref{solver} that \tblPBS can indeed effectively fix variables and prune the nodes in the search tree; overall, it fixes an average of $51.3$ variables per node and prunes $59$ nodes per instance. Due to this, \tblPBS significantly outperforms \tblRF and \tblMIX.}
In particular, \rev{compared with \tblRF (which is relatively faster than \tblMIX), the proposed \tblPBS can solve \rev{$75$} more instances to optimality within the time limit of $4$ hours;}
the average number of explored nodes is reduced by a factor of $5.5$; 
and the average CPU time is reduced by a factor of $\rev{3.4}$. 
Moreover, for hard instances that require to explore at least $1000$ nodes by at least one of the \rev{three} settings, 
we can observe a drastic node reduction factor of $\rev{10.3}$ and a drastic runtime speed-up factor of $\rev{5.8}$.
In addition, among the three problems, we can observe a tremendous improvement on \mpp instances where the tree size constructed by \tblRF is relatively large. 
Indeed, using the proposed \tblPBS, the average number of explored nodes and the average CPU time decrease from $\rev{4916}$ and $\rev{539.1}$ seconds to $\rev{161}$ and \rev{$53.5$} seconds, respectively, with $39$ more solved instances.
These results show the \rev{effectiveness} of the proposed dominance-based branching (\rev{with the} overlap-oriented node pruning and variable fixing techniques) in reducing the search tree size and improving the efficiency of solving \CCPs, and \rev{the efficiency of the proposed approaches over the state-of-the-art approaches in \citet{Ruszczynski2002} and \citet{Luedtke2010a}}.

\begin{table}[htbp]
	\small
	\addtolength{\tabcolsep}{-2.2pt}
	\centering
	\caption{\rev{Performance comparison of the proposed dominance-based branching (with the overlap-oriented node pruning and variable fixing techniques) with the state-of-the-art approaches in \citet{Ruszczynski2002} and \citet{Luedtke2010a}.}}
	\begin{tabular*}{\textwidth}{@{\extracolsep\fill}lrrrrrrrrrrrrr@{\extracolsep\fill}}
		\toprule
		\multirow{2}{*}{\PROB} & \multirow{2}{*}{\NUM} &
		\multicolumn{3}{@{}c@{}}{\tblMIX} & \multicolumn{3}{@{}c@{}}{\tblRF} & \multicolumn{6}{@{}c@{}}{\tblPBS} \\
		\cmidrule(l{4pt}r{3pt}){3-5} \cmidrule(l{4pt}r{3pt}){6-8} \cmidrule(l{4pt}r){9-14}
		& &\tblS & \tblT & \tblN &\tblS & \tblT & \tblN & \tblS & \tblT & \tblN & \tblTd & \tbldomred & \tblcutoff  \\ 
		\midrule
		\tblRPP  &119 & 101 & 339.5 & 230 & 102 & 398.0 & 226 & \textbf{119} & \textbf{160.8} & \textbf{59}& 9.5 & 42.5 & 1 \\ 
		\tblMPP &122 & 83 & 538.4 & 4937 & 83 & 539.1 & 4916 & \textbf{122} & \textbf{53.5} & \textbf{161}& 2.9 & 99.7 & 31 \\ 
		\tblLSP &164 & 136 & 413.0 & 8285 & 145 & 307.7 & 3617 & \textbf{164} & \textbf{162.3} & \textbf{1266}& 25.8 & 28.5 & 153 \\ 
		\midrule
		\texttt{All} &405 & 320 & 422.3 & 2681 & 330 & 393.0 & 1890 & \textbf{405} & \textbf{116.0} & \textbf{341}& 10.4 & 51.3 & 59 \\ 
		\texttt{$\geq10$} & 348 & 263 & 761.9 & 4669 & 273 & 687.4 & 3133 & \textbf{348} & \textbf{163.4} & \textbf{460}& 12.5 & 56.8 & 71 \\ 
		\texttt{$\geq100$} & 315 & 230 & 1087.3 & 6835 & 240 & 943.4 & 4438 & \textbf{315} & \textbf{198.5} & \textbf{564}& 14.3 & 60.4 & 81 \\ 
		\texttt{$\geq1000$} & 244 & 159 & 1985.7 & 16702 & 169 & 1564.9 & 9801 & \textbf{244} & \textbf{271.1} & \textbf{955}& 18.5 & 66.1 & 116 \\ 
		\bottomrule
	\end{tabular*}
	\label{solver}
\end{table}

\subsection{Comparison of the dominance-based branching with the direct use of dominance inequalities}\label{subsect:branch}

\rev{To gain more insight into the proposed dominance-based branching, here we compare it with the direct use of dominance inequalities. 
	To do this, we compare the following two settings}:
\begin{itemize}
	\item [$\bullet$] \tblFBS: \rev{solving formulation \eqref{milp} using the proposed dominance-based branching (but without the overlap-oriented node pruning and variable fixing techniques) where the mixing cuts were also implemented.}
	\item [$\bullet$] \tblSRF: \rev{solving formulation \eqref{milp} with the stronger version of dominance inequalities in \eqref{newprecedenceconstraints} (in the sense that inequalities in \eqref{newprecedenceconstraints} imply all inequalities in \eqref{precedenceconstraints}) using  the \BnC algorithm, where the mixing cuts were also implemented.}
\end{itemize}
\rev{For benchmarking purposes, we also report the results of \tblRF where the weaker version of dominance inequalities in \eqref{precedenceconstraints} were added.}

\rev{\cref{branch} summarizes the computational results of settings \tblFBS, \tblSRF, and \tblRF.}
For settings  \tblSRF and \tblRF, we additionally report the percentage of dominance pairs between scenarios and the percentage of nonredundant dominance inequalities, defined by $\rev{\tblDR} := \DP/\MDP \times 100$ and $\tblTDR:= \NDI/\MDP \times 100$, respectively.
Here, \DP is the number of dominance pairs,  
$\MDP=n(n-1)/2$ is the maximum number of possible dominance pairs, 
and \NDI is the number of nonredundant dominance inequalities.

\begin{table}[htbp]
	\small
	\centering
	\addtolength{\tabcolsep}{-3pt}
	\caption{Performance comparison of settings \tblFBS, \tblSRF, and \tblRF.}
	\begin{tabular*}{\textwidth}{@{\extracolsep\fill}lrrrrrrrrrrrrrrrr@{\extracolsep\fill}}
		\toprule
		\multirow{2}{*}{\PROB} & \multirow{2}{*}{\NUM}
		& \multicolumn{3}{c}{\tblFBS} & \multicolumn{5}{c}{\tblSRF} & \multicolumn{5}{c}{\tblRF} \\ 
		\cmidrule(l{4pt}r{3pt}){3-5}\cmidrule(l{4pt}r{3pt}){6-10} \cmidrule(l{4pt}r{3pt}){11-15} &
		& \tblS & \tblT & \tblN & \tblS & \tblT & \tblN & \tblDR & \tblTDR & \tblS & \tblT & \tblN & \tblDR & \tblTDR \\
		\midrule 
		\tblRPP &112 & \textbf{110} & \textbf{175.9} & 117  & 108 & 225.5 & \textbf{90}  & 51.4 & 1.2  & 102 & 317.8 & 211  & 5.0 & 4.3 \\ 
		\tblMPP &110 & \textbf{110} & \textbf{47.6} & 301  & \textbf{110} & 56.5 & \textbf{261}  & 40.3 & 0.2  & 83 & 376.5 & 4175  & 0.1 & 0.1 \\ 
		\tblLSP  &160 & \textbf{158} & 169.6 & 2410  & \textbf{158} & \textbf{163.0} & \textbf{1744}  & 90.9 & 0.4  & 145 & 279.5 & 3296  & 35.5 & 7.2 \\ 
		\midrule
		\texttt{All} &382 & \textbf{378} & \textbf{119.1} & 622  & 376 & 132.3 & \textbf{492}  & 63.2 & 0.6  & 330 & 316.2 & 1701  & 15.2 & 4.3 \\ 
		\bottomrule
	\end{tabular*}
	\label{branch}
\end{table}

First, we can observe from \cref{branch} that the number of dominance pairs in \eqref{newprecedenceconstraints} is much larger than that in \eqref{precedenceconstraints}, 
which shows the effectiveness of the preprocessing technique in \cref{subsect:preprocess} in detecting more dominance pairs.
Although the number of dominance pairs in \eqref{newprecedenceconstraints} is very large, only a small proportion of the inequalities (i.e., the nonredundant ones) needs to be added into the formulation;
in some cases, the number of nonredundant inequalities in \eqref{newprecedenceconstraints} is even smaller than that of the nonredundant ones in \eqref{precedenceconstraints}.
Due to these two advantages, \tblSRF performs much better than \tblRF.
In particular, using \tblSRF, $\rev{46}$ more instances can be solved to optimality, and the average CPU time and the average number of explored nodes are reduced by factors of $\rev{2.4}$ and $\rev{3.5}$, respectively.

Next, we compare the performance of settings \tblFBS and \tblSRF.
As shown in \cref{subsect:existing work}, under certain conditions, the two settings are theoretically equivalent in terms of exploring the same search trees.
In \cref{branch}, we observe that the numbers of explored nodes returned by the two settings are similar;
overall, \tblSRF returns a slightly smaller number of explored nodes than \tblFBS. 
Notice that the latter is reasonable as directly adding the dominance inequalities into the formulation can enhance other components of \MILP solvers, 
such as triggering more internal cuts generation and more conflict analysis \citep{Achterberg2007a,Witzig2021a}.
However, \rev{for \rpp and \mpp instances}, the smaller number of explored nodes returned \rev{by} \tblSRF cannot compensate for the overhead of a large problem size, as the dominance inequalities in \eqref{newprecedenceconstraints} need to be added as constraints into the formulation.
In contrast, the proposed \tblFBS avoids solving a problem with a large problem size and thus achieves an overall better performance;
\rev{overall, the proposed \tblFBS enables us to reduce the solution times of \rpp and \mpp instances by factors of $1.3$ and $1.2$, respectively.
Note that for \lsp instances, \tblSRF can return a fairly smaller number of nodes than that returned by \tblFBS, and thus performs slightly better than \tblFBS.}

\subsection{Performance effect of the overlap-oriented node pruning and variable fixing}
\label{subsect:fixing}

We now evaluate the performance effect of the overlap-oriented node pruning and variable fixing in \cref{sect:consistency branch}.
\cref{pruneandfix} summarizes the computational results of settings \rev{\tblFBS and \tblPBS.}
As shown in \cref{pruneandfix}, the overlap-oriented node pruning and variable fixing can \rev{fix a relatively large number of variables at the nodes and prune a fairly large number of nodes}, while the computational overhead of the approximation approach for implementing them is fairly small.
In total, equipped with the overlap-oriented node pruning and variable fixing, \tblPBS can solve \rev{$27$} more instances to optimality.
Additionally, the average CPU time and the average number of explored nodes are reduced by factors of \rev{$1.4$} and \rev{$2.2$}, respectively.
For hard instances that require to explore at least $1000$ nodes by at least one of the two settings, 
we can even observe a factor of $\rev{1.8}$ runtime speed-up and a factor of $\rev{3.5}$ tree size reduction. 
This shows the effectiveness of the overlap-oriented node pruning and variable fixing in further improving the performance of the dominance-based branching for solving \CCPs.

\begin{table}[htbp]
	 \small
	\centering
	\addtolength{\tabcolsep}{-3pt}
	\caption{Performance comparison of settings \tblFBS and \tblPBS.}
	\begin{tabular*}{\textwidth}{@{\extracolsep\fill}lrrrrrrrrrrrr@{\extracolsep\fill}}
		\toprule
		\multirow{2}{*}{\PROB} & \multirow{2}{*}{\NUM} &
		\multicolumn{5}{c}{\tblFBS} & \multicolumn{6}{c}{\tblPBS} \\
		\cmidrule{3-7} 	\cmidrule{8-13} 
		& & \tblS & \tblT & \tblN &\tblTd & \tbldomred & \tblS & \tblT & \tblN &\tblTd & \tbldomred & \tblcutoff \\ 
		\midrule
		\tblRPP&119 & 110 & 228.2 & 135 & \textbf{1.5} & 7.4 & \textbf{119} & \textbf{160.8} & \textbf{59} & 9.5 & 42.5 & 1 \\ 
		\tblMPP &122 & 110 & 84.1 & 496 & \textbf{1.5} & 16.6  & \textbf{122} & \textbf{53.5} & \textbf{161} & 2.9 & 99.7 & 31 \\
		\tblLSP  &164 & 158 & 189.1 & 2700 & \textbf{2.6} & 14.8  & \textbf{164} & \textbf{162.3} & \textbf{1266} & 25.8 & 28.5 & 153 \\
		\midrule
		\texttt{All} &405 & 378 & 156.6 & 749 & \textbf{1.9} & 13.1  & \textbf{405} & \textbf{116.0} & \textbf{341} & 10.4 & 51.3 & 59 \\ 
		\texttt{$\geq10$} & 291 & 264 & 367.3 & 1834 & \textbf{2.8} & 17.4  & \textbf{291} & \textbf{232.8} & \textbf{679} & 16.6 & 57.0 & 91 \\ 
		\texttt{$\geq100$} & 233 & 206 & 638.0 & 3660 & \textbf{3.7} & 18.0  & \textbf{233} & \textbf{368.1} & \textbf{1158} & 23.5 & 59.9 & 124 \\ 
		\texttt{$\geq1000$} & 164 & 137 & 1181.5 & 9194 & \textbf{5.5} & 19.0 & \textbf{164} & \textbf{648.6} & \textbf{2636} & 39.8 & 62.3 & 212 \\ 
		\bottomrule
	\end{tabular*}
	\label{pruneandfix}
\end{table}

In order to illustrate the effectiveness of the proposed approximation approach in \cref{subsect:polynomial branch} for implementing the overlap-oriented node pruning and variable fixing, we compare it with the exact  approach in \cref{subsect:consistency analysis} (denoted by \tblENPF).
In \cref{anpfandenpf}, we report the computational results of \rev{$287$} instances that can be solved to optimality in both  \tblPBS and \tblENPF within $24$ hours. 
We observe that \tblPBS is much more efficient than \tblENPF.
This is quite expected because \tblENPF needs to solve several \MILP problems of the form \eqref{MIP} to perform 
\rev{the overlap-oriented variable fixing,}
which is very time-consuming, as shown in column \tblTd under setting \tblENPF.
On the other hand, 
\rev{using the exact approach, the average number of fixed variables per node returned by \tblENPF is only slightly larger than that of \tblPBS}\footnote{\rev{Note that no overlap-oriented node pruning can be performed in \tblENPF, as discussed in \cref{re:nofixing}.}}, 
\rev{and thus the average number of explored nodes returned by \tblENPF is only slightly smaller than that of \tblPBS.}
These results demonstrate that the proposed approximation algorithm for implementing the overlap-oriented node pruning and variable fixing is competitive with the exact algorithm in terms of reducing the tree size while enjoying a high computational efficiency.

\begin{table}[htbp]
	 \small
	\centering
	\addtolength{\tabcolsep}{-3pt}
	\caption{Performance comparison of settings \tblPBS and \tblENPF.}
	\begin{tabular*}{\textwidth}{@{\extracolsep\fill}lrrrrrrrrrrrrr@{\extracolsep\fill}}
		\toprule
		\multirow{2}{*}{\PROB} & \multirow{2}{*}{\NUM} &
		\multicolumn{6}{c}{\tblPBS} & \multicolumn{6}{c}{\tblENPF} \\
		\cmidrule(l{4pt}r{3pt}){3-8} \cmidrule(l{4pt}r{3pt}){9-14} &
		& \tblS & \tblT & \tblN &\tblTd & \tbldomred & \tblcutoff  & \tblS & \tblT & \tblN &\tblTd & \tbldomred  & \tblcutoff \\ 
		\midrule
		\tblRPP  & 89 & \textbf{89} & \textbf{61.5} & \textbf{16} & \textbf{6.8} & 30.8 & $<$1  & \textbf{89} & 877.3 & \textbf{16} & 661.9 & \textbf{31.5} & 0 \\ 
		\tblMPP & 93 & \textbf{93} & \textbf{20.4} & 46 & \textbf{1.5} & 79.7 & 5  & \textbf{93} & 293.8 & \textbf{43} & 233.3 & \textbf{79.9} & 0 \\ 
		\tblLSP &105 & \textbf{105} & \textbf{46.2} & 231 & \textbf{7.3} & 25.9 & 17  & \textbf{105} & 1164.8 & \textbf{217} & 990.5 & \textbf{26.3} & 0 \\ 
		\midrule
		\texttt{All} &287 & \textbf{287} & \textbf{38.8} & 83 & \textbf{4.5} & 43.0 & 7  & \textbf{287} & 682.9 & \textbf{80} & 547.3 & \textbf{43.4} & 0 \\ 
		\bottomrule
	\end{tabular*}	\label{anpfandenpf}
\end{table}

\rev{\subsection{Performance effect of increasing the dimension $m$ of the random vector} \label{subsect:effectm}}

\rev{In this subsection, we compare the performance of the proposed dominance-based branching (with the overlap-oriented node pruning and variable fixing techniques) with different dimensions $m$ of the random vector $\xi$.}
\rev{\cref{dimension} summarizes the computational results for settings \tblMIX, \tblSRF, \tblFBS, and \tblPBS on \mpp instances.
In the table, column \DeltaS denotes the difference in the number of solved instances returned by each of the three settings (i.e., \tblSRF, \tblFBS, and \tblPBS) and \tblMIX.
A positive value under the three settings means that \tblSRF, \tblFBS, and \tblPBS can solve more instances to optimality than \tblMIX.
Columns \RT and \RN display the ratios of the average CPU time and the average number of explored nodes.
A value greater than $1.0$ represents an improvement for \tblSRF, \tblFBS, and \tblPBS.}

\rev{We first evaluate the performance of \tblSRF and \tblFBS with the increase of $m$. 
	It is worth noting that as the dimension $m$ increases, we should expect that the performance improvement  of the proposed {\tblSRF and}  \tblFBS over \tblMIX tends to decline as the percentage of dominance pairs decreases.
	This trend is clearly illustrated in \cref{dimension}.
	In particular, for instances with $m = 150$, the percentage of dominance pairs is nearly zero, and consequently, the improvement of the proposed {\tblSRF and} \tblFBS over \tblMIX becomes neutral.
	In contrast, for instances with $m = 10$ where the percentage of dominance pairs is $55.9\%$, we can observe a significant performance improvement of {\tblSRF and} \tblFBS over \tblMIX.
	Nevertheless, even for instances with $m = 100$, where the percentage of dominance pairs is $1.2\%$, the proposed {\tblSRF and} \tblFBS still enable us to reduce the average CPU time by {factors of $1.1$ and} $1.2$, {respectively}.}

\rev{Next, we evaluate the performance of \tblPBS with the increase of $m$.  
	We can observe from \cref{dimension} that similar to \tblSRF and \tblFBS, the performance improvement of the proposed \tblPBS over \tblMIX tends to decline as the dimension $m$ increases.
	However, even for instances with $m = 150$, where the percentage of dominance pairs is nearly zero, the proposed \tblPBS still enables us to reduce the average CPU time and the average number of explored nodes by factors of $1.2$ and $1.3$, respectively,
	showing that the overlap-oriented node pruning and variable fixing techniques can even improve the performance of solving \CCPs without the presence of the dominance pairs. This is reasonable, as it does not directly rely on the number of dominance pairs.}

\begin{table}[htbp]
	\small
	\centering
	\caption{\rev{Performance comparison of settings \tblMIX, \tblSRF, \tblFBS, and \tblPBS with different dimensions $m$ of the random vector $\xi$ on the \tblMPP instances.}}
	\addtolength{\tabcolsep}{1pt}
	\begin{tabular*}{\textwidth}{@{\extracolsep\fill}llrrrrrr@{\extracolsep\fill}}
		\toprule
		& $m$ & 10 & 20 & 30 & 50 & 100 & 150 \\
		\midrule
		& \tblDR & 55.9 & 29.7 & 21.3 & 11.1 & 1.2 & 0.1 \\
		\midrule
		\tblMIX
		& \tblS & 34 & 27 & 22 & 17 & 15 & 12 \\
		& \tblT & 337.2 & 660.0 & 772.8 & 679.8 & 1005.1 & 494.6 \\
		& \tblN & 1804 & 1670 & 2372 & 1727 & 2096 & 567 \\
		\midrule
		\tblSRF
		& \DeltaS & 11 & 7 & 9 & 2 & 0 & 0 \\
		& \RT & 6.1 & 5.0 & 5.5 & 3.8 & 1.1 & 1.0 \\
		& \RN & 30.8 & 22.7 & 15.8 & 5.5 & 1.3 & 1.0 \\
		\midrule
		\tblFBS 
		& \DeltaS & 11 & 7 & 9 & 3 & 1 & 0 \\
		& \RT & 8.5 & 5.2 & 5.8 & 3.4 & 1.2 & 1.1 \\
		& \RN & 26.6 & 16.8 & 14.3 & 5.2 & 1.1 & 1.0 \\
		\midrule
		\tblPBS 
		& \DeltaS & 11 & 16 & 12 & 5 & 2 & 1 \\
		& \RT & 12.2 & 8.9 & 9.2 & 4.8 & 2.4 & 1.2 \\
		& \RN & 89.9 & 87.2 & 38.6 & 16.6 & 3.3 & 1.3 \\
		\bottomrule
	\end{tabular*}\label{dimension}
\end{table}

\section{Conclusion}
\label{sect:Conclusion and future work}

In this paper, it has been shown that the presence of the overlap in the search tree makes the standard \MILP formulation of the \CCP difficult to be solved by state-of-the-art \MILP solvers.
In an effort to remedy this, we have developed several approaches to remove the overlap during the \BnC process.
In particular, we have showed that a family of valid nonlinear if-then constraints is able to remove all overlaps in the search tree. 
To tackle the highly nonlinear if-then constraints, we have proposed the dominance-based branching, which, compared with the classic variable branching, is able to partition the current problem into two \MILP  subproblems with much smaller feasible regions, especially when the number of dominance pairs between scenarios is large.  
Moreover, by considering the joint mixing set with a knapsack constraint and the if-then constraints, we have developed overlap-oriented node pruning and variable fixing, applied at each node of the search tree, to remove more overlaps in the tree.
The proposed dominance-based branching \rev{with the} overlap-oriented node pruning and variable fixing can be easily embedded into the \BnC framework along with other sophisticated components of \MILP solvers.
By extensive computational experiments, we have demonstrated that the proposed dominance-based branching \rev{with the} overlap-oriented node pruning and variable fixing techniques can significantly reduce the \BnC search tree size and substantially improve the computational performance of the \MILP-based approach to solving \CCPs.

\begin{appendices}
	\section{An example to illustrate \cref{remark:left}}\label{sec:example}
	In this section, we present an example to illustrate \cref{remark:left}.
	The example is based on the problem obtained by modifying scenario $\xi^6$ to $(6,\,2,\,5)^\T$ in problem \eqref{expro1}. 
	As such, formulation \eqref{milp} for the modified problem reduces to 
	\begin{equation}\small
		\label{ex:append}
		\min \left\{6v_1 + v_2 + 3 v_3 : v \geq \xi^i (1-z_i),\,\forall\,i \in [7],~\frac{1}{7}\sum_{i =1}^7 z_i \leq \frac{4}{7},~v \in \R_+^3,~z \in \{0,1\}^7 \right\},
	\end{equation}
	where
	\begin{equation*}
		\xi^{1} = \left(\begin{array}{c}  2 \\ 1 \\ 12 \\ \end{array}\right),\,
		\xi^{2} = \left(\begin{array}{c}  3 \\ 1 \\ 10 \\ \end{array}\right),\,
		\xi^{3} = \left(\begin{array}{c}  4 \\ 2 \\ 7 \\ \end{array}\right),\,
	\end{equation*}
	\begin{equation*}
		\xi^{4} = \left(\begin{array}{c}  5 \\ 2 \\ 6 \\	\end{array}\right),\,
		\xi^{5} = \left(\begin{array}{c}  6 \\ 2 \\ 6 \\ \end{array}\right),\,
		\xi^{6} = \left(\begin{array}{c}  6 \\ 2 \\  5 \\ \end{array}\right),\,
		\xi^{7} = \left(\begin{array}{c}  12 \\ 1 \\ 2 \\ \end{array}\right).
	\end{equation*}
	
	We choose variable $z_4$ to branch on. 
	The left and right branches $($after removing the fixed variable $z_4$$)$ reduce to
	\begin{equation}\label{exbranchL}\small
		\begin{aligned}
			& O^{\rm L} = \min \left\{6v_1 + v_2 + 3 v_3 : \frac{1}{7}(z_1+z_2+z_3+z_5+z_6+z_7) \leq \frac{4}{7},\,v \geq \xi^4=(5, 2, 6)^\T,\right.\\
			& \qquad\qquad\qquad \qquad\qquad \left.v \geq \xi^i(1-z_i),~\forall~i\in\{1,2,3,5,6,7\},~v\in\R_+^3,~z \in \{0,1\}^6 \vphantom{\frac{1}{7}}\right\},
		\end{aligned}
	\end{equation}
	\begin{equation}\label{exbranchR}\small
		\begin{aligned}
			& O^{\rm R} = \min \left\{6v_1 + v_2 + 3 v_3 \stt \frac{1}{7}(z_1+z_2+z_3+z_5+z_6+z_7) \leq \frac{3}{7},~v \geq \boldsymbol{0},\right.\\
			& \qquad\qquad \qquad\qquad\quad\left. ~v \geq \xi^i(1-z_i),~\forall~i\in\{1,2,3,5,6,7\},~v\in\R_+^3,~z \in \{0,1\}^6 \vphantom{\frac{1}{7}} \right\}.
		\end{aligned}
	\end{equation}		
	By \cref{remark:left}, the new left branch $O^{\rm L_2}$ reduces to
	\begin{equation}\label{exbranchL2}\small
		\begin{aligned}
			& \min \left\{6v_1 + v_2 + 3 v_3 \stt \frac{3}{7} < \frac{1}{7}(z_1+z_2+z_3+z_5+z_6+z_7) \leq \frac{4}{7},\,v \geq \xi^4=(5, 2, 6)^\T, \right.\\
			& \qquad\qquad\qquad \qquad\qquad \left. v \geq \xi^i(1-z_i),~\forall~i\in\{1,2,3,5,6,7\},~v \in \R_+^3,~z \in \{0,1\}^6 \vphantom{\frac{1}{7}}\right\}.
		\end{aligned}
	\end{equation}
	Although the feasible regions of branches \eqref{exbranchR} and \eqref{exbranchL2} do not contain an overlap, their projection on the $v$ space, however, do contain overlaps.
	Indeed, point $v^*=(6,\,2,\,7)^\T$ can define feasible solutions 
	$$
	(v^*,\hat{z}) = (6,\,2,\,7,\,1,\,1,\,0,\,0,\,0,\,1)^\T \text{and}~ (v^*,\bar{z}) = (6,\,2,\,7,\,1,\,1,\,0,\,0,\,1,\,1)^\T
	$$
	of the two branches \eqref{exbranchR} and \eqref{exbranchL2}, respectively.
	
	\section{Test problems}\label{sect:appendix}
	
	Here we introduce the \rpp, \mpp, and \lsp problems considered in the computational experiments and the details of the procedures to generate the instances.
	
	\subsection{Chance-constrained resource planning problem \citep{Luedtke2014,Gurvich2010}}
	
	Consider a set of resources $\I$ with unit cost $c_i$ for each $i \in \I$ and a set of customers  $\J$ with random demand $\tilde{\lambda}_j$ for each $j \in \J$. 
	The \rpp problem attempts to choose the quantities of the resources and allocate these resources to customers such that the total cost of resources is minimized,
	while requiring that the allocation does not exceed the available resources and meets the customer demand with a probability at least $1-\epsilon$.
	
	Let $x_i$ denote the quantity of resource $i \in \I$ and $y_{ij}$ denote the amount of resource $i$ allocated to customer $j \in \J$. 
	The mathematical formulation of this problem can be stated as:
	\begin{equation*}
		\!\!\min_{(x,y) \in \R^{|\I|+|\I||\J|}_+} \! \left\{ \sum_{i \in \I} c_i x_i : \sum_{j \in \J} y_{ij} \leq \rho_i x_i, \forall\,i \in \I ,
		\,{\mathbb{P}}\left\{\sum_{i \in \I} \mu_{ij}y_{ij}\geq \tilde{\lambda}_j, \forall\,j \in \J \right\} \!\geq 1-\epsilon \right\},
	\end{equation*}
	 where $\rho_i \in (0,1]$ is the yield of resource $i$ and $\mu_{ij} \geq 0$ is the service rate of resource $i$ to customer $j$.
	We use the data available at \url{https://jrluedtke.github.io}; see \cite{Luedtke2014} for the detailed generation procedure. 
	The numbers of resources and customers, represented as $(|\I|, |\J|)$, are chosen from $\{(20, 30), (40, 50), (50, 100)\}$.
	The scenario size $n$ for the random variables $\{\tilde{\lambda}_j\}$ and the confidence parameter $\epsilon$ are taken from $\{1000, 2000, 3000\}$ and $\{0.10, 0.15, 0.20\}$, respectively.
	The probabilities $p_i$, $i \in [n]$, are all set to $1/n$. 
	For each $(|\I|, |\J|)$, $n$, and $\epsilon$, there are $5$ instances, leading to a testbed of $135$ \rpp instances.

	\subsection{Chance-constrained multiperiod power planning problem \cite{Dey2023}}
	The \mpp problem attempts to expand the electric power capacity of a state by building new coal and nuclear power plants to meet the demand of the state over a planning horizon of $T$ periods.
	Coal and nuclear plants are operational for $T_c$ and $T_n$ time periods after their construction, respectively.
	The objective of the problem is to minimize the total capital cost associated with the construction of the power plants while requiring that 
	the fraction of nuclear capacity to the total capacity does not exceed a predetermined threshold $f$ (as required by legal restrictions mandate), and the demands are met with a probability at least $1-\epsilon$.
	
	Let $x_t$ and $y_t$ be the amount of coal and nuclear capacity (in megawatt) brought on line at the beginning of period $t$, respectively.
	Then, the mathematical formulation of this problem is given by 
	\begin{equation*}
		\begin{aligned}
			& \min_{(x, y)\in \R_+^{2T}} \left \{ \sum_{t \in [T]} (c_t x_t +n_t y_t) \stt 
			\sum^t_{i = \tau_n(t)} y_i \leq f \cdot \left(e_t + \sum^t_{i = \tau_c(t)} x_i + \sum^t_{i = \tau_n(t)} y_i \right), \,\forall~t \in [T], \right.\\
			& \left.\qquad\qquad\qquad\qquad\qquad\qquad~~ {\mathbb{P}} \left\{e_t + \sum^t_{i = \tau_c(t)} x_i + \sum^t_{i = \tau_n(t)} y_i  \geq \tilde{d}_t,~\forall~t \in [T] \right\} \geq 1-\epsilon \right\},
		\end{aligned}
	\end{equation*}
	where $c_t$ and $n_t$ are the capital costs per megawatt for coal  and nuclear power plants, respectively, 
	$e_t$ is the electric capacity from existing resources in period $t$,
	$\tau_c(t) = \max \{1,\,t-T_c+1\}$, and $\tau_n(t) = \max \{1,\,t-T_n+1 \}$, $t \in [T]$.
	We use a similar procedure of \cite{Dey2023} to construct \mpp instances.
	Specifically, 
	\begin{itemize}
		\item [$\bullet$] the electricity demands $\{\tilde{d}_t\}$ are independent random integers, and their scenarios are uniformly chosen from $\{300, 301, \ldots,700\}$;
		\item [$\bullet$] the costs $c_t$ and $n_t$ are uniformly chosen from $\{100, 101, \ldots,300\}$ and $\{100, 101, \ldots,200\}$, respectively;
		\item [$\bullet$] the initial capacity resource $e_1$ is an integer uniformly chosen from $\{100, 101, \ldots,500\}$ and the capacity resources in the subsequent periods are calculated as $e_t= e_1 \cdot r^{t-1}$, where $t=2, 3, \ldots, T$, and $r$ is uniformly chosen from  $[0.7,1)$;
		\item [$\bullet$] the lifespans of coal and nuclear power plants, $T_c$ and $T_n$, are set to $15$ and $10$, respectively;
		\item [$\bullet$] the number of periods $T$ is taken from $\{10,20,30\}$;
		\item [$\bullet$] the scenario size $n$ for the random variables $\{\tilde{d}_t\}$ and the confidence parameter $\epsilon$ are taken from $\{1000,2000,3000\}$ and $\{0.05, 0.10, 0.20\}$, respectively;
		\item [$\bullet$] the probabilities $p_i$, $i \in [n]$, are all set to $1/n$. 
	\end{itemize}
	For each $T$, $n$, and $\epsilon$, we randomly generate $5$ instances, and thus in total, there are $135$ \mpp instances.
	\rev{Note that in the experiments conducted in \cref{subsect:effectm}, to compare the performance of the proposed approaches on instances with different dimensions of the random vector $\xi$, the number of periods $T$ is taken from $\{10, 20, 30, 50, 100, 150\}$, while the other parameters are still identical (and thus, there are in total $270$ \mpp instances considered in \cref{subsect:effectm}).}
	
	\subsection{Chance-constrained lot-sizing problem \citep{Beraldi2002}}\label{cclsp}
	The \lsp problem attempts to determine a production schedule for $T$ periods that minimizes the sum of the fixed setup cost, the production cost, and the expected inventory cost while  
	satisfying the capacity constraint in each period and meeting the customer demand with a probability at least $1-\epsilon$.
	
	Let $x_t$ be the binary variable denoting whether a setup of production is performed in period $t$, and $y_t$ be the continuous variable characterizing the corresponding quantity to be produced.
	Mathematically, the \lsp problem can be written by 
	\begin{equation*}
		\begin{aligned}
			& \min _{(x, y) \in \{0,1\}^T \times \R_{+}^T } \left\{\sum_{t \in[T]}\left(f_t x_t+ c_t y_t + h_t \mathbb{E}\left(\left(\sum_{j \in [t]} y_j-\sum_{j \in[t]} \tilde{d}_j\right)^+ \right)\right):\right. \\
			&\left. \qquad\qquad\qquad\qquad y_t \leq C_t x_t,\,\forall\,t \in[T], \, \mathbb{P}\left\{\sum_{j \in[t]} y_j \geq \sum_{j \in[t]} \tilde{d}_j, \,\forall\, t \in[T]\right\} \geq 1-\epsilon\right\},
		\end{aligned}
	\end{equation*}
	where for each $t\in [T]$, $f_t$ is the fixed setup cost per production run, $c_t$ is the unit production cost, 
	$h_t$ is the unit holding cost, $C_t$ is the production capacity, and $\tilde{d}_t$ is the random demand.
	Notice that the \lsp problem can also be transformed into an \MILP problem of the form \eqref{milp}; see \cite{Zhao2017}.
	We use a similar procedure as in \cite{Abdi2016}  to construct the \lsp instances.
	Specifically,
	\begin{itemize}
		\item [$\bullet$] the demands $\{\tilde{d}_{j}\}$ are independent random variables, and their scenarios are uniformly chosen from $(1, 100)$;
		\item [$\bullet$] the setup cost $f_t$, the unit production cost $c_t$, and the unit holding cost $h_t$ are uniformly chosen from $(1, 1000)$, $(1, 10)$, and $(1, 5)$, respectively;
		\item [$\bullet$] the number of periods $T$ is taken from $\{5, 10, 15, 20\}$;
		\item [$\bullet$] the scenario size $n$ for the random variables $\{\tilde{d}_j\}$ and the confidence parameter $\epsilon$ are taken from $\{1000, 2000, 3000\}$ and $\{0.05, 0.10, 0.20\}$, respectively;
		\item [$\bullet$]  the probabilities $p_i$, $i \in [n]$, are all set to $1/n$. 
	\end{itemize}
	For each $T$, $n$, and $\epsilon$, we randomly generate $5$ instances, and thus in total, there are $180$ \lsp instances.

\end{appendices}

\bibliography{shorttitles,CCP-arXiv}{}

\begin{thebibliography}{}

\bibitem[Abdi and Fukasawa, 2016]{Abdi2016}
Abdi, A. and Fukasawa, R. (2016).
\newblock On the mixing set with a knapsack constraint.
\newblock {\em Math. Program.}, 157:191--217.

\bibitem[Achterberg, 2007a]{Achterberg2007a}
Achterberg, T. (2007a).
\newblock {Conflict analysis in mixed integer programming}.
\newblock {\em Discrete Optim.}, 4(1):4--20.

\bibitem[Achterberg, 2007b]{Achterberg2007}
Achterberg, T. (2007b).
\newblock {\em {Constraint Integer Programming}}.
\newblock Ph.{D}. thesis, Technical University of Berlin.

\bibitem[Achterberg et~al., 2005]{Achterberg2005}
Achterberg, T., Koch, T., and Martin, A. (2005).
\newblock Branching rules revisited.
\newblock {\em Oper. Res. Lett.}, 33(1):42--54.

\bibitem[Ahmed et~al., 2017]{Ahmed2017}
Ahmed, S., Luedtke, J., Song, Y., and Xie, W. (2017).
\newblock Nonanticipative duality, relaxations, and formulations for
  chance-constrained stochastic programs.
\newblock {\em Math. Program.}, 162(1):51--81.

\bibitem[Ahmed and Shapiro, 2008]{Ahmed2008}
Ahmed, S. and Shapiro, A. (2008).
\newblock Solving chance-constrained stochastic programs via sampling and
  integer programming.
\newblock {\em Tutor. Oper. Res. (INFORMS)}, 10:261--269.

\bibitem[Atamt{\"{u}}rk et~al., 2000]{Atamturk2000a}
Atamt{\"{u}}rk, A., Nemhauser, G.~L., and Savelsbergh, M.~W. (2000).
\newblock The mixed vertex packing problem.
\newblock {\em Math. Program.}, 89:35--53.

\bibitem[Barrera et~al., 2016]{Barrera2016}
Barrera, J., Homem-de Mello, T., Moreno, E., Pagnoncelli, B.~K., and Canessa,
  G. (2016).
\newblock {Chance-constrained problems and rare events: {A}n importance
  sampling approach}.
\newblock {\em Math. Program.}, 157(1):153--189.

\bibitem[Beraldi and Bruni, 2010]{Beraldi2010}
Beraldi, P. and Bruni, M.~E. (2010).
\newblock An exact approach for solving integer problems under probabilistic
  constraints with random technology matrix.
\newblock {\em Ann. Oper. Res.}, 177(1):127--137.

\bibitem[Beraldi et~al., 2004]{Beraldi2004}
Beraldi, P., Bruni, M.~E., and Conforti, D. (2004).
\newblock {Designing robust emergency medical service via stochastic
  programming}.
\newblock {\em Eur. J. Oper. Res.}, 158(1):183--193.

\bibitem[Beraldi and Ruszczy\'{n}ski, 2002]{Beraldi2002a}
Beraldi, P. and Ruszczy\'{n}ski, A. (2002).
\newblock The probabilistic set-covering problem.
\newblock {\em Oper. Res.}, 50(6):956--967.

\bibitem[Beraldi and Ruszczyński, 2002]{Beraldi2002}
Beraldi, P. and Ruszczyński, A. (2002).
\newblock A branch and bound method for stochastic integer problems under
  probabilistic constraints.
\newblock {\em Optim. Methods Softw.}, 17(3):359--382.

\bibitem[Bestuzheva et~al., 2021]{solverscip2021}
Bestuzheva, K., Besan{\c{c}}on, M., Chen, W.-K., Chmiela, A., Donkiewicz, T.,
  van Doornmalen, J., Eifler, L., Gaul, O., Gamrath, G., Gleixner, A.,
  Gottwald, L., Graczyk, C., Halbig, K., Hoen, A., Hojny, C., van~der Hulst,
  R., Koch, T., L{\"u}bbecke, M., Maher, S.~J., Matter, F., M{\"u}hmer, E.,
  M{\"u}ller, B., Pfetsch, M.~E., Rehfeldt, D., Schlein, S., Schl{\"o}sser, F.,
  Serrano, F., Shinano, Y., Sofranac, B., Turner, M., Vigerske, S.,
  Wegscheider, F., Wellner, P., Weninger, D., and Witzig, J. (2021).
\newblock {The SCIP Optimization Suite 8.0}.
\newblock Technical report, Optimization Online.

\bibitem[Bolusani et~al., 2024]{solverscip2024}
Bolusani, S., Besan{\c{c}}on, M., Bestuzheva, K., Chmiela, A., Dion{\'{i}}sio,
  J., Donkiewicz, T., van Doornmalen, J., Eifler, L., Ghannam, M., Gleixner,
  A., Graczyk, C., Halbig, K., Hedtke, I., Hoen, A., Hojny, C., van~der Hulst,
  R., Kamp, D., Koch, T., Kofler, K., Lentz, J., Manns, J., Mexi, G.,
  M\"{u}hmer, E., Pfetsch, M.~E., Schl{\"o}sser, F., Serrano, F., Shinano, Y.,
  Turner, M., Vigerske, S., Weninger, D., and Xu, L. (2024).
\newblock {The SCIP Optimization Suite 9.0}.
\newblock Technical report, Optimization Online.

\bibitem[Charnes and Cooper, 1963]{Charnes1963}
Charnes, A. and Cooper, W.~W. (1963).
\newblock Deterministic equivalents for optimizing and satisficing under chance
  constraints.
\newblock {\em Oper. Res.}, 11(1):18--39.

\bibitem[Chen et~al., 2023]{Chen2023}
Chen, L., Chen, S.-J., Chen, W.-K., Dai, Y.-H., Quan, T., and Chen, J. (2023).
\newblock Efficient presolving methods for solving maximal covering and partial
  set covering location problems.
\newblock {\em Eur. J. Oper. Res.}, 311(1):73--87.

\bibitem[Chen et~al., 2024]{Chen2024}
Chen, W.-K., Chen, Y.-L., Dai, Y.-H., and Lv, W. (2024).
\newblock Towards large-scale probabilistic set covering problem: {A}n
  efficient {B}enders decomposition approach.
\newblock In {\em 2024 Informs Optimization Society conference}.

\bibitem[Cheon et~al., 2006]{Cheon2006}
Cheon, M.-S., Ahmed, S., and Al-Khayyal, F. (2006).
\newblock A branch-reduce-cut algorithm for the global optimization of
  probabilistically constrained linear programs.
\newblock {\em Math. Program.}, 108(2):617--634.

\bibitem[Dentcheva et~al., 2004]{Dentcheva2004}
Dentcheva, D., Lai, B., and Ruszczy{\'{n}}ski, A. (2004).
\newblock {Dual methods for probabilistic optimization problems}.
\newblock {\em Math. Methods Oper. Res.}, 60(2):331--346.

\bibitem[Dentcheva et~al., 2000]{Dentcheva2000}
Dentcheva, D., Pr{\'{e}}kopa, A., and Ruszczy{\'{n}}ski, A. (2000).
\newblock {Concavity and efficient points of discrete distributions in
  probabilistic programming}.
\newblock {\em Math. Program.}, 89(1):55--77.

\bibitem[Dentcheva et~al., 2001]{Dentcheva2001}
Dentcheva, D., Pr{\'{e}}kopa, A., and Ruszczy{\'{n}}ski, A. (2001).
\newblock {On convex probabilistic programming with discrete distributions}.
\newblock {\em Nonlinear Anal.}, 47(3):1997--2009.

\bibitem[Dey et~al., 2024]{Dey2023}
Dey, S.~S., Dubey, Y., Molinaro, M., and Shah, P. (2024).
\newblock A theoretical and computational analysis of full strong-branching.
\newblock {\em Math. Program.}, 205(1):303--336.

\bibitem[Gamrath, 2014]{Gamrath2014}
Gamrath, G. (2014).
\newblock Improving strong branching by domain propagation.
\newblock {\em EURO J. Comput. Optim.}, 2(3):99--122.

\bibitem[Gamrath, 2020]{Gamrath2020a}
Gamrath, G. (2020).
\newblock {\em {Enhanced Predictions and Structure Exploitation in
  Branch-and-Bound}}.
\newblock Ph.{D}. thesis, Technical University of Berlin.

\bibitem[Garey and Johnson, 1978]{Garey1978}
Garey, M.~R. and Johnson, D.~S. (1978).
\newblock ``{S}trong'' {NP}-completeness results: Motivation, examples, and
  implications.
\newblock {\em J. ACM}, 25(3):499–508.

\bibitem[G{\"{u}}nl{\"{u}}k and Pochet, 2001]{Gunluk2001b}
G{\"{u}}nl{\"{u}}k, O. and Pochet, Y. (2001).
\newblock Mixing mixed-integer inequalities.
\newblock {\em Math. Program.}, 90(3):429--457.

\bibitem[Gurvich et~al., 2010]{Gurvich2010}
Gurvich, I., Luedtke, J., and Tezcan, T. (2010).
\newblock Staffing call centers with uncertain demand forecasts: A
  chance-constrained optimization approach.
\newblock {\em Manag. Sci.}, 56(7):1093--1115.

\bibitem[Henrion and R{\"o}misch, 2022]{Henrion2022}
Henrion, R. and R{\"o}misch, W. (2022).
\newblock Problem-based optimal scenario generation and reduction in stochastic
  programming.
\newblock {\em Math. Program.}, 191(1):183--205.

\bibitem[Jiang and Xie, 2025]{Jiang2025}
Jiang, N. and Xie, W. (2025).
\newblock The terminator: An integration of inner and outer approximations for
  solving wasserstein distributionally robust chance constrained programs via
  variable fixing.
\newblock {\em INFORMS J. Comput.}, 37(2):381--412.

\bibitem[Koch et~al., 2011]{MIPLIB2010}
Koch, T., Achterberg, T., Andersen, E., Bastert, O., Berthold, T., Bixby,
  R.~E., Danna, E., Gamrath, G., Gleixner, A.~M., Heinz, S., Lodi, A.,
  Mittelmann, H., Ralphs, T., Salvagnin, D., Steffy, D.~E., and Wolter, K.
  (2011).
\newblock {MIPLIB} 2010.
\newblock {\em Math. Program. Comput.}, 3(2):103--163.

\bibitem[K{\"{u}}{\c{c}}{\"{u}}kyavuz, 2012]{Kucukyavuz2012}
K{\"{u}}{\c{c}}{\"{u}}kyavuz, S. (2012).
\newblock {On mixing sets arising in chance-constrained programming}.
\newblock {\em Math. Program.}, 132:31--56.

\bibitem[K{\"{u}}{\c{c}}{\"{u}}kyavuz and Jiang, 2022]{Kucukyavuz2022}
K{\"{u}}{\c{c}}{\"{u}}kyavuz, S. and Jiang, R. (2022).
\newblock {Chance-constrained optimization under limited distributional
  information: A review of reformulations based on sampling and distributional
  robustness}.
\newblock {\em EURO J. Comput. Optim.}, 10:100030.

\bibitem[Kılın{\c{c}}-Karzan et~al., 2022]{Klnc-Karzan2022}
Kılın{\c{c}}-Karzan, F., K{\"{u}}{\c{c}}{\"{u}}kyavuz, S., and Lee, D.
  (2022).
\newblock {Joint chance-constrained programs and the intersection of mixing
  sets through a submodularity lens}.
\newblock {\em Math. Program.}, 195(1):283--326.

\bibitem[Lejeune and Noyan, 2010]{Lejeune2010}
Lejeune, M. and Noyan, N. (2010).
\newblock Mathematical programming approaches for generating $p$-efficient
  points.
\newblock {\em Eur. J. Oper. Res.}, 207(2):590--600.

\bibitem[Lejeune, 2008]{Lejeune2008}
Lejeune, M.~A. (2008).
\newblock Preprocessing techniques and column generation algorithms for
  stochastically efficient demand.
\newblock {\em J. Oper. Res. Soc.}, 59(9):1239--1252.

\bibitem[Lejeune, 2012a]{Lejeune2012a}
Lejeune, M.~A. (2012a).
\newblock Pattern-based modeling and solution of probabilistically constrained
  optimization problems.
\newblock {\em Oper. Res.}, 60(6):1356--1372.

\bibitem[Lejeune, 2012b]{Lejeune2012b}
Lejeune, M.~A. (2012b).
\newblock Pattern definition of the $p$-efficiency concept.
\newblock {\em Ann. Oper. Res.}, 200(1):23--36.

\bibitem[Lejeune and Ruszczy{\'{n}}ski, 2007]{Lejeune2007}
Lejeune, M.~A. and Ruszczy{\'{n}}ski, A. (2007).
\newblock An efficient trajectory method for probabilistic
  production-inventory-distribution problems.
\newblock {\em Oper. Res.}, 55(2):378--394.

\bibitem[Luedtke, 2014]{Luedtke2014}
Luedtke, J. (2014).
\newblock {A branch-and-cut decomposition algorithm for solving
  chance-constrained mathematical programs with finite support}.
\newblock {\em Math. Program.}, 146:219--244.

\bibitem[Luedtke and Ahmed, 2008]{Luedtke2008}
Luedtke, J. and Ahmed, S. (2008).
\newblock A sample approximation approach for optimization with probabilistic
  constraints.
\newblock {\em SIAM J. Optim.}, 19(2):674--699.

\bibitem[Luedtke et~al., 2010]{Luedtke2010a}
Luedtke, J., Ahmed, S., and Nemhauser, G.~L. (2010).
\newblock {An integer programming approach for linear programs with
  probabilistic constraints}.
\newblock {\em Math. Program.}, 122:247--272.

\bibitem[Margot, 2002]{Margot2002}
Margot, F. (2002).
\newblock {Pruning by isomorphism in branch-and-cut}.
\newblock {\em Math. Program.}, 94(1):71--90.

\bibitem[Miller and Wagner, 1965]{Miller1965}
Miller, B.~L. and Wagner, H.~M. (1965).
\newblock Chance constrained programming with joint constraints.
\newblock {\em Oper. Res.}, 13(6):930--945.

\bibitem[Murr and Pr{\'e}kopa, 2000]{Murr2000}
Murr, M.~R. and Pr{\'e}kopa, A. (2000).
\newblock Solution of a product substitution problem using stochastic
  programming.
\newblock In Uryasev, S.~P., editor, {\em Probabilistic Constrained
  Optimization: Methodology and Applications}, pages 252--271. Springer US,
  Boston, MA.

\bibitem[Nemirovski and Shapiro, 2006]{Nemirovski2006a}
Nemirovski, A. and Shapiro, A. (2006).
\newblock Scenario approximations of chance constraints.
\newblock In Calafiore, G. and Dabbene, F., editors, {\em Probabilistic and
  Randomized Methods for Design under Uncertainty}, pages 3--47. Springer,
  London.

\bibitem[Ostrowski et~al., 2011]{Ostrowski2011}
Ostrowski, J., Linderoth, J., Rossi, F., and Smriglio, S. (2011).
\newblock Orbital branching.
\newblock {\em Math. Program.}, 126(1):147--178.

\bibitem[Pfetsch and Rehn, 2019]{Pfetsch2019}
Pfetsch, M.~E. and Rehn, T. (2019).
\newblock A computational comparison of symmetry handling methods for mixed
  integer programs.
\newblock {\em Math. Program. Comput.}, 11(1):37--93.

\bibitem[Pr{\'{e}}kopa, 1971]{Prekopa1970}
Pr{\'{e}}kopa, A. (1971).
\newblock On probabilistic constrained programming.
\newblock In {\em Proceedings of the Princeton Symposium on Mathematical
  Programming}, pages 113--138. Princeton University Press, Princeton.

\bibitem[Pr{\'{e}}kopa, 1973]{Prekopa1973}
Pr{\'{e}}kopa, A. (1973).
\newblock {Contributions to the theory of stochastic programming}.
\newblock {\em Math. Program.}, 4(1):202--221.

\bibitem[Pr{\'e}kopa, 1990]{Prekopa1990}
Pr{\'e}kopa, A. (1990).
\newblock Dual method for the solution of a one-stage stochastic programming
  problem with random rhs obeying a discrete probability distribution.
\newblock {\em Z. Oper. Res.}, 34(6):441--461.

\bibitem[Pr{\'e}kopa, 2003]{Prekopa2003}
Pr{\'e}kopa, A. (2003).
\newblock Probabilistic programming.
\newblock In Ruszczy{\'{n}}ski, A. and Shapiro, A., editors, {\em Stochastic
  Programming}, volume~10 of {\em Handbooks Oper. Res. Management Sci.}, pages
  267--351. Elsevier, Amsterdam.

\bibitem[Qiu et~al., 2014]{Qiu2014}
Qiu, F., Ahmed, S., Dey, S.~S., and Wolsey, L.~A. (2014).
\newblock {Covering linear programming with violations}.
\newblock {\em INFORMS J. Comput.}, 26(3):531--546.

\bibitem[Qiu and Wang, 2015]{Qiu2015}
Qiu, F. and Wang, J. (2015).
\newblock Chance-constrained transmission switching with guaranteed wind power
  utilization.
\newblock {\em IEEE Trans. Power Syst.}, 30(3):1270--1278.

\bibitem[Ruszczy{\'{n}}ski, 2002]{Ruszczynski2002}
Ruszczy{\'{n}}ski, A. (2002).
\newblock {Probabilistic programming with discrete distributions and precedence
  constrained knapsack polyhedra}.
\newblock {\em Math. Program.}, 93(2):195--215.

\bibitem[Saxena et~al., 2010]{Saxena2010}
Saxena, A., Goyal, V., and Lejeune, M.~A. (2010).
\newblock \rm{MIP} reformulations of the probabilistic set covering problem.
\newblock {\em Math. Program.}, 121(1):1--31.

\bibitem[Sen, 1992]{Sen1992}
Sen, S. (1992).
\newblock Relaxations for probabilistically constrained programs with discrete
  random variables.
\newblock {\em Oper. Res. Lett.}, 11(2):81--86.

\bibitem[Song and Luedtke, 2013]{Song2013}
Song, Y. and Luedtke, J.~R. (2013).
\newblock {Branch-and-cut approaches for chance-constrained formulations of
  reliable network design problems}.
\newblock {\em Math. Program. Comput.}, 5(4):397--432.

\bibitem[Song et~al., 2014]{Song2014}
Song, Y., Luedtke, J.~R., and K{\"{u}}{\c{c}}{\"{u}}kyavuz, S. (2014).
\newblock {Chance-constrained binary packing problems}.
\newblock {\em INFORMS J. Comput.}, 26(4):735--747.

\bibitem[Vielma et~al., 2012]{Vielma2012}
Vielma, J.~P., Ahmed, S., and Nemhauser, G.~L. (2012).
\newblock Mixed integer linear programming formulations for probabilistic
  constraints.
\newblock {\em Oper. Res. Lett.}, 40(3):153--158.

\bibitem[Wang et~al., 2012]{Wang2012}
Wang, Q., Guan, Y., and Wang, J. (2012).
\newblock A chance-constrained two-stage stochastic program for unit commitment
  with uncertain wind power output.
\newblock {\em IEEE Trans. Power Syst.}, 27(1):206--215.

\bibitem[Witzig et~al., 2021]{Witzig2021a}
Witzig, J., Berthold, T., and Heinz, S. (2021).
\newblock {Computational aspects of infeasibility analysis in mixed integer
  programming}.
\newblock {\em Math. Program. Comput.}, 13(4):753--785.

\bibitem[Zhao et~al., 2017]{Zhao2017}
Zhao, M., Huang, K., and Zeng, B. (2017).
\newblock {A polyhedral study on chance constrained program with random
  right-hand side}.
\newblock {\em Math. Program.}, 166(1-2):19--64.

\end{thebibliography}
\bibliographystyle{apalike}

\end{document}